\documentclass[11pt,oneside,a4paper]{amsart}\openup1.6\jot
\usepackage{amsthm}
\usepackage{amsmath}
\usepackage{amsrefs} 
\usepackage{amssymb}
\usepackage{amsfonts}
\usepackage{amscd}
\usepackage{mathrsfs} 
\usepackage{mathtools}
\usepackage[dvipdfmx]{graphicx}
\usepackage{bm}   
\usepackage[all]{xy}
\usepackage{ulem}
\usepackage[all]{xy}
\usepackage{ulem}
\newtheorem{theorem}{Theorem}[section]
\newtheorem{lemma}[theorem]{Lemma}
\newtheorem{proposition}[theorem]{Proposition}
\newtheorem{corollary}[theorem]{Corollary}

\newtheorem{assumption}[theorem]{Assumption}

\theoremstyle{definition}
\newtheorem{definition}[theorem]{Definition}

\newtheorem{remark}[theorem]{Remark} 

\def\C{\mathbb{C}} 
\def\R{\mathbb{R}}

\def\P{\mathbb{P}}
\def\Z{\mathbb{Z}}

\def\i{{\mathrm{i}}}

\def\id{\mathrm{id}}

\newcommand{\End}{\mathrm{End}}
\newcommand{\Hom}{\mathrm{Hom}}

\newcommand{\homscr}{{\mathscr{H}\! \! om}}
\def\Cok{{\mathrm{Coker}}}

\def\({\left(}
\def\){\right)}
\def\<{\langle}
\def\>{\rangle}
\newcommand{\simeqto}{\xrightarrow{\sim}}
\newcommand{\jump}[1]{\ensuremath{[\![#1]\!]} }
\newcommand{\pole}[1]{\ensuremath{(\!(#1)\!)} }
\newcommand{\conv}[1]{\ensuremath{(\!\{#1\}\!)}}
\newcommand{\cal}[1]{\mathcal{#1}}
\renewcommand{\scr}[1]{\mathscr{#1}}

\def\Ker{{\mathrm{Ker}}}

\def\O{{\mathscr{O}}}
\def\per{{\mathrm{per}}}

\def\asy{{\mathrm{asy}}}
\def\gr{{\mathrm{gr}}}

\def\RH{{\mathrm{RH}}}
\def\DR{{\mathrm{DR}}}
\def\Sol{{\mathrm{Sol}}}

\def\aut{{\scr{A}ut}}

\def\dR{{\mathsf{dR}}}
\def\Be{{\mathsf{Be}}}
\def\rh{{\mathrm{rh}}}
\def\gl{{\mathrm{global}}}
\def\loc{{\mathrm{local}}}
\def\Cpt{{\mathrm{Cpt}}}
\def\Exp{{\mathrm{Exp}}}
\def\cpx{{\mathsf{cpx}}}
\def\CS{{\mathrm{CS}}}

\title[A RH correspondence for cohomology theories of closed 1-forms
]{A Riemann--Hilbert correspondence for cohomology theories of closed 1-forms
}
\author{Yota Shamoto}
\address{Faculty of Science and Engineering, Yamato University, Suita, Osaka 564-0082, Japan.
} 
\email{shamoto.yota@yamato-u.ac.jp}

\begin{document}
\begin{abstract}
Motivated by the work of Kontsevich--Soibelman on the comparison of isomorphisms conjecture for closed algebraic $1$-forms, we establish a Riemann--Hilbert correspondence of Deligne--Malgrange type.

As an application, we prove a variant of the comparison of isomorphisms theorem for a simple class of algebraic $1$-forms on complex curves.
\end{abstract}
\maketitle

\section{Introduction}
In this paper, we reinterpret a conjecture of Kontsevich--Soibelman \cite{kontsevich2024holomorphic}*{Conjecture 4.7.1}, which is known as the comparison of isomorphisms conjecture, from the perspective of the Deligne--Malgrange Riemann--Hilbert correspondence \cites{Deligne,Malgrange}.
We begin in \S\ref{Intro CIC} by recalling the part of \cite{kontsevich2024holomorphic} concerning the comparison of isomorphisms conjecture. 
After the preparatory discussions in \S\ref{Intro Periods}, \S \ref{Main idea}, and \S \ref{Intro Coeff}, we formulate our Riemann--Hilbert correspondence in \S\ref{Intro RH}. Finally, in \S\ref{Intro 1dim}, we describe our formulation of a one-dimensional version of the comparison of isomorphisms theorem.

\subsection{Comparison of isomorphisms conjecture}
\label{Intro CIC}

Let $X$ be a complex projective manifold, and let $\alpha$ be a meromorphic closed 1-form on $X$ whose polar divisor is supported on a hypersurface $D \subset X$.

In their study of holomorphic Floer theory and its relation to exponential integrals, Kontsevich--Soibelman associate with a pair $(X,\alpha)$ (in a more general setting) four $z$-dependent cohomology theories:
\begin{align*}
H_{\dR,\gl}^{*}(X,\alpha),&&
H_{\dR,\loc}^{*}(X,\alpha),&& 
H_{\Be,\gl}^{*}(X,\alpha),&&
H^*_{\Be,\loc}(X,\alpha)
\end{align*}
called the global and local de Rham cohomology, and the global and local Betti cohomology, respectively.

The role of the parameter $z$ varies among these theories: 
it is an algebraic parameter in the global de Rham theory, 
a formal variable in the local de Rham theory, 
a complex analytic parameter in the global Betti theory,
and a parameter admitting a meromorphic extension in the local Betti theory. 
We briefly recall their definitions.

The global de Rham cohomology $H_{\dR,\gl}^{*}(X,\alpha)$ 
is defined as the hypercohomology of the complex
\begin{align}\label{Intro dR global}
    (\Omega_X^*(*D)[z],\, zd+\alpha\wedge),
\end{align}
where $z$ is viewed as an algebraic parameter.

The local de Rham cohomology $H_{\dR,\loc}^{*}(X,\alpha)$ 
is defined as the hypercohomology of the formally completed complex
\begin{align}\label{Intro localdR}
    (\Omega_X^*(*D)\jump{z},\, zd+\alpha\wedge),
\end{align}
where $z$ is regarded as a formal variable.

For each $z\in\C^*=\C\setminus\{0\}$, let
\[
\nabla^z = d + z^{-1}\alpha.
\]
The fiber $H_{\Be,\gl}^{*}(X,\alpha)_z$ 
of the global Betti cohomology 
is defined as the rapid decay cohomology of $\nabla^z$.
As $z$ varies in $\C^*$, these groups form a coherent sheaf on $\C^*$,
which is denoted by 
$H_{\Be,\gl}^{*}(X,\alpha)$.

The local Betti cohomology $H_{\Be,\loc}^{*}(X,\alpha)$ 
is defined analytically in $z$ as the direct sum of relative homologies: 
Let $Z(\alpha)=\bigcup_{j}Z_j$
be irreducible decomposition. We set  
\begin{align}\label{Intro Betti local}
    \mathcal{H}_{\Be,\loc,z}^{\bullet}(X,\alpha)\coloneqq 
\bigoplus_{j}H^\bullet (U_j,U_j\cap f_j^{-1}(\varepsilon \cdot e^{\i\theta});\Z),
\end{align}
where $U_j$ is an sufficiently small
neighborhood of
the component $Z_j$,
$f_j\colon U_j\to \C$
is a holomorphic function satisfying 
$df_j=\alpha$ and $f_{j|Z_j}=0$,
$\varepsilon>0$ is sufficiently small 
positive real number,
and $\theta=\arg(z)$.
We have a local system
$\mathcal{H}_{\Be,\loc}^{\bullet}(X,\alpha)$ of $\Z$-modules
on $\C^*$.
We set \[H^*_{\Be,\loc}(X,\alpha)=(j_*\mathcal{H}_{\Be,\loc}(X,\alpha)\otimes_{j_*\Z_{\C^*}} 
\mathscr{O}_\C(*0))_0,\]
where $j\colon \C^* \to \C$ denotes the inclusion, and 
$\mathscr{O}_\C(*0)$
denote the sheaf of meromorphic
functions on $\C$ whose poles are contained
in $\{0\}$.
The subscript $0$ on the right hand side
denotes 
the stalk at $0$.
It is defined over 
the field ${\C\conv{z}}$
of convergent Laurent series in $z$.

It is also observed that the global 
Betti cohomology has the extension 
to the module
$H^*_{\Be,\gl}(X,\alpha)^{\rm mero}$
over $\C\conv{z}$, 
constructed by the theory of wall-crossing 
structure. The module 
$H^*_{\Be,\gl}(X,\alpha)^{\rm mero}$ is 
called the meromorphic extension. 

On the one hand, 
we have isomorphisms 
\begin{align*}
    &\RH_\gl\colon 
    H^*_{\dR,\gl}(X,\alpha)\otimes_{\C[z]}
    \mathscr{O}_{\C^*}
    \longrightarrow 
    H^*_{\Be,\gl}(X,\alpha),\text{ and }\\
    &\RH_\loc\colon H^*_{\dR,\loc}(X,\alpha)\otimes_{\C\jump{z}}\C\pole{z}\longrightarrow 
    H_{\Be,\loc}(X,\alpha)\otimes_{{\C\conv{z}}}\C\pole{z}. 
\end{align*}
Here, $\scr{O}_{\C^*}$ denotes 
the sheaf of 
holomorphic functions on $\C^*$. 
On the other hand, there are 
so called local-to-global
isomorphisms
\begin{align*}
&{\phi}_{\dR}\colon H^*_{\dR,\gl}(X,\alpha)\otimes\C\pole{z}
\longrightarrow H^*_{\dR,\loc}(X,\alpha)\otimes_{\C\jump{z}}\C\pole{z},\text{ and }\\
&\phi_\Be \colon H^*_{\Be,\gl}(X,\alpha)^{\rm mero}\otimes_{{\C\conv{z}}}\C\pole{z}
\longrightarrow H^*_{\Be,\loc}(X,\alpha)
\otimes_{{\C\conv{z}}} \C\pole{z}.
\end{align*}
Here, $\phi_\dR$ is given by the degeneration of spectral sequence
for formal completion.  
The isomorphism $\phi_\Be$ 
is defined by  
the theory of analytic 
wall-crossing structure.  

Then, they conjectured 
\cite{kontsevich2024holomorphic}*{Conjecture 4.7.1}
that 
\begin{itemize}
\item $\RH_{\gl}$ has meromorphic extension
\[\RH_{\gl}^{\rm mero}\colon 
H^*_{\dR,\gl}(X,\alpha)\otimes_{\C[z]}{\C\conv{z}}
\to 
H^*_{\Be,\gl}(X,\alpha)^{\rm mero}\]
to $z=0$, and that 
\item the diagram
\begin{align}\begin{split}\label{KS}
    \xymatrix@C=5em{
    H^*_{\dR,\gl}(X,\alpha)\otimes_{\C[z]} \C\pole{z}
    \ar[r]^{{\RH}_\gl^{\rm formal}}\ar[d]_{\phi_\dR}
    &H^*_{\Be,\gl}(X,\alpha)^{\rm mero}\otimes \C\pole{z}
    \ar[d]^{\phi_\Be}\\
    H^*_{\dR,\loc}(X,\alpha)\otimes_{\C\jump{z}}\C\pole{z}
    \ar[r]^{\RH_\loc}
    &H^*_{\Be,\loc}(X,\alpha)\otimes_{{\C\conv{z}}} \C\pole{z}
    }\end{split}
\end{align}
commutes, where ${\RH}_\gl^{\rm formal}$
denotes the formal completion
of ${\RH}_\gl^{\rm mero}$.
\end{itemize}
We refer to this conjecture as 
\textit{the comparison of 
isomorphisms conjecture}.

In this paper, we do not address this conjecture directly. 
Instead, we notice the following fact: 
After they proposed the conjecture,
they indicated the relation to the Riemann--Hilbert correspondence of Deligne--Malgrange type
\cite{kontsevich2024holomorphic}*{Remark 4.7.2}:
\begin{quote}
        The Conjecture $4.7.1$ can be thought of as a version of the Deligne--Malgrange Riemann--Hilbert correspondence in the case of irregular $D$-modules associated with closed $1$-forms. 
\end{quote}
The purpose of this paper is to provide a formulation that makes this remark precise.

\subsection{Background 
of the remark}\label{Intro Periods}
If the $1$-form $\alpha$ is exact, 
that is, if there exists a 
meromorphic function $f$ on $X$
with poles in $D$
such that $\alpha=df$,
then one obtains a 
globally defined connection 
in the $z$-direction:
$\nabla=d+d(f/z)$.
It follows that the global de Rham cohomology
is naturally equipped with 
a connection 
in the $z$-direction.
Let $(E,\nabla)$
denote the associated
germ of a meromorphic connection 
over $\C\conv{z}$. 
One can apply the 
Deligne--Malgrange 
Riemann--Hilbert functor \cite{Deligne}
to $(E,\nabla)$
to obtain 
a Stokes-filtered local system
on the circle. 
The underlying local system 
can be identified with the
global Betti cohomology. 

When $\alpha$ is not exact, however, 
there is in general no globally defined connection in the $z$-direction.
Kontsevich--Soibelman \cite{kontsevich2024holomorphic}
overcome this difficulty 
as follows. 
In \cite{kontsevich2024holomorphic}*{Section 2},
they reinterpret the Stokes structure of 
exponential integrals associated with $f$, 
i.e., the Stokes-filtered local system 
associated with $(E,\nabla)$ 
in terms of wall-crossing structure.
They generalize 
the wall-crossing structure to 
the non-exact case in \cite{kontsevich2024holomorphic}*{Section 3}. 
This viewpoint appears to underlie their remark
\cite{kontsevich2024holomorphic}*{Remark 4.7.1}. 
Their approach is ambitious, encompassing infinite-dimensional situations; however, its relation to the classical Riemann--Hilbert correspondence becomes less transparent.

\subsection{Main idea of our approach}\label{Main idea}
In our approach,
we propose a more direct generalization of
the notion of 
Stokes-filtered local systems
by replacing the coefficient field
$\C$ by a filtered sheaf of rings
on the circle.
Correspondingly,
we also generalize the de Rham side, i.e., 
the notion of germs of meromorphic connections in the $z$-variable. 
Before going into the details, 
we first explain the main idea 
using the simplest non-trivial example 
related to the Gamma function,
which illustrates the key phenomena.
This example is also studied in  
\cite{kontsevich2024holomorphic}*{Section 4}.

Let $X=\mathbb{P}^1$
be a projective line with 
affine coordinate $x$.
Set $D=\{0,\infty\}$.
Let us consider the meromorphic closed one form $\alpha=-(1-x)x^{-1}dx$
on $X$ with poles in $D$. 
Although the $1$-form
$\alpha$
is not exact, there exists 
a multivalued function 
$f=x-\log x$
with $df=\alpha$. 
The difference between two branches of 
$\log x$ is given by the period
$\int_{|x|=1}\alpha=-2\pi \i$.

On the de Rham side, 
the isomorphism $\phi_{\dR}$ can be described explicitly   
in terms of the Borel summable formal power series
\[\exp\left(-\sum_{n=1}^\infty \frac{B_{2n}}{2n(2n-1)}z^{2n-1}\right),\]
where $B_{2n}$ denotes the Bernoulli numbers. 
By the Borel-Laplace 
transform, 
one obtains two connections 
\begin{align*}
    \nabla&=d+(\psi(1/z)-\log(1/z)-1)\frac{dz}{z^2},\\
    \nabla'&= d+(\psi(1-z^{-1})-\log(1/z)
    -\pi \i-1)\frac{dz}{z^2},
\end{align*}
where $\psi(s)$ denotes the digamma function. 
The connections $\nabla$ and $\nabla'$  are analytic lifts 
of the formal connection 
  $  \widehat{\nabla}=d-2^{-1}\left( 1+\sum_{n=1}^\infty {B_{2n}}{n}^{-1}z^{2n-1}\right)z^{-1}{dz}-z^{-2}{dz}$ over the 
sectors $\{-\pi<\arg(1/z)< \pi\}$
and $\{0<\arg(1/z)<2\pi\}$,
respectively.
If we set $u=\exp(-2\pi \i /z)$ and $g=1-u$,
we have $g\nabla=\nabla'g$
on the sector $\{0<\arg(1/z)<\pi\}$.
A similar relation holds on 
the sector $\{-\pi<\arg(1/z)<0\}$.

On the Betti side, 
although 
$x-\log x$
is multivalued, 
\[\C[u,u^{-1}]x^{1/z}\exp(-x/z),\quad  
u=\exp(-2\pi\i/z)\]
is a well-defined subsheaf 
of the sheaf $\scr{O}_{\C^*\times \C^*}$,
which defines a rank-one 
local system of 
$\C[u,u^{-1}]$-modules.
The variable $u$ 
encodes the monodromy 
arising from the multivaluedness of $ x^{1/z}$.
Integrating the $1$-form $x^{-1}dx$ along the rapid decay homology classes 
of this local system,
one obtains functions
$z^{1/z}\Gamma(1/z)$ on the sector $\{-\pi<\arg(1/z)< \pi\}$
and 
$z^{1/z}\Gamma(1-z^{-1})^{-1}$
on the sector $\{0<\arg(1/z)<2\pi\}$,
respectively. 
They are also connected by $g$
on $\{0<\arg(1/z)<\pi\}$,
which is known as
the reflection formula for the 
Gamma function. 
A similar relation holds on 
the sector $\{-\pi<\arg(1/z)<0\}$.
See \S \ref{Example G} for details
of our approach to this object.

This example suggests the following:
\begin{itemize}
    \item On the de Rham side,
    germs of meromorphic connections should be
    generalized to connections
    on sectors, which have the same asymptotic 
    expansions. These should be related to each other by functions like 
    $u=\exp(-2\pi \i/z)$. 
    \item On the Betti side,
    Stokes-filtered local system 
    should be generalized 
    to modules over a ring
    which contains a generalization of $u=\exp(-2\pi\i/z)$.
\end{itemize}
On both sides, the function 
$u$ 
plays a central role. It is induced from the period $\int_{|x|=1}\alpha$,
which encodes the multivaluedness 
of the primitive $x-\log x$ of $\alpha$. 

In general, returning to the setting in \S \ref{Intro CIC},
the multivaluedness of 
a primitive of the meromorphic closed $1$-form 
$\alpha$ on the complex projective manifold $X$
whose poles are supported on $D$ is encoded by 
 the period map
 \begin{align*}
     \widetilde{\mu}_\alpha\colon 
     H_1(Y,\Z)\to\C, \gamma\mapsto \int_\gamma\alpha,
    \quad Y=X\setminus D. 
 \end{align*}
Since only the periods appearing in the Stokes matrices are relevant, we may discard the kernel of $\widetilde{\mu}_\alpha$.
We therefore set
\begin{align}\label{intromu}
    L=H_1(Y,\Z)/\Ker(\widetilde{\mu}_\alpha),
\quad \mu=\mu_\alpha\colon L\to\C
\end{align}
for the induced injective homomorphism.
The function $u$ is 
replaced by functions of the form
\begin{align*}
    \exp(\mu(\gamma)/z) \quad \gamma\in L.
\end{align*} 

This suggests that both sides of 
the correspondence 
should be formulated over a 
common sheaf of rings 
containing these exponential functions.

\subsection{The coefficient sheaf $\scr{A}_\mu$}\label{Intro Coeff}
We now introduce the coefficient sheaf $\scr{A}_\mu$ on a
circle $S^1=\{e^{\i\theta}\mid \theta\in\R\}$, which encodes the exponential growth behavior governed by the period homomorphism $\mu$.
The sheaf $\scr{A}_\mu$
will play a central role in our theory. 

Let $\varpi\colon \widetilde{\C}\to \C$
be the real oriented blow-up of the origin in the complex plane. Let $z$ denote the coordinate function.
We regard $S^1$ as the boundary of $\widetilde{\C}$. In particular, 
$\theta$ denotes the argument of $z$. 
Let $\scr{O}_{\C^*}$ denote the 
sheaf of holomorphic functions on $\C^*=\C\setminus \{0\}$.
Let $\imath\colon S^1\to \widetilde{\C}$
and $\jmath\colon \C^*\to \widetilde{\C}$
denote inclusions, and set
    $\widetilde{\scr{O}}
    =\imath^{-1}\jmath_*
    \scr{O}_{\C^*}. $
Let $\scr{A}^{\leqslant 0}\subset\widetilde{\scr{O}}$
 be the subsheaf consisting of functions of moderate growth as 
$z\to 0$ (see \S\ref{sheaves on S^1} for a precise definition).
Let
$L\simeq \Z^r$ be a lattice and let
$\mu\colon L\to \C$
be a homomorphism.
We then define the sheaf 
$\scr{A}_\mu$
as follows.  
We first define a subsheaf  $\scr{A}_\mu^{\leqslant 0}\subset \scr{A}_\mu$. 
For an open subset 
$I\subset S^1$ a section of $\scr{A}_\mu^{\leqslant 0}$
over $I$ is a convergent 
series of the form 
\begin{align*}
    h(z)=\sum_{\gamma\in L}a_\gamma\exp(\mu(\gamma)/z).
\end{align*}
which defines a holomorphic function on the corresponding sector of $\widetilde{\C}$
and has moderate growth as $z\to 0$ along directions in $I$,
i.e., $h(z)\in \scr{A}^{\leqslant 0}(I)$.
The convergence of the above infinite series is not automatic.
To obtain a workable sufficient condition ensuring convergence and good topological properties of the resulting ring, we impose the support property \eqref{support property} on the homomorphism 
$\mu$.
(See \S\ref{Amu} for the precise formulation.)
Having defined the moderate part, we set
\[\scr{A}_\mu=\sum_{\gamma\in L }\exp(\mu(\gamma)/z)\scr{A}_\mu^{\leqslant 0},\]
which defines a subsheaf of rings in 
$\widetilde{\scr{\O}}$. 

There is a natural filtration 
$\scr{A}_\mu^{\leqslant c} $ 
indexed by $c\in \C$
which measures the exponential growth of sections relative to $\exp(c/z)$ along 
sectors (see \S \ref{Amu} for details). 
The filtered sheaf of rings 
$\scr{A}_\mu$
will serve as the common coefficient structure in our formulation of the Riemann--Hilbert correspondence.
In particular, it enters both the de Rham and the Betti sides of the correspondence.

\subsection{Riemann--Hilbert correspondence of Deligne--Malgrange type}
\label{Intro RH}
Now we formulate a Riemann--Hilbert correspondence adapted to the exponential data encoded by the coefficient sheaf $\scr A_\mu$ introduced in \S\ref{Intro Coeff}. 
This correspondence is governed by a pair $(C,\mu)$ consisting of exponential and period data.

In addition to the lattice homomorphism $\mu:L\to\C$, 
we fix a finite subset $C\subset\C$, 
which indexes the formal exponential factors $e^{c/z}$ 
appearing in the irregular connections under consideration. 

On the de Rham side, 
we consider a category $\dR_{C,\mu}$.
An object of $\dR_{C,\mu}$ consists of a tuple
\[
(E,\mathcal{I},\Xi,(E^0,\nabla^0)),
\]
where:

\begin{itemize}
\item $E$ is a finite dimensional $\C\conv{z}$-vector space;
\item $(E^0,\nabla^0)$ is the stalk at $0$ of a meromorphic connection,
      whose decomposition 
      $E^0=\bigoplus_{c\in C}E^0_c,$
      satisfies the condition that 
      $\nabla^0_{|E^0_c}-d(c/z)$ is regular singular;

\item $(\mathcal{I},\Xi)$ is additional sectorial data,
      encoding analytic lifts and their transition isomorphisms,
      to be described below.
\end{itemize}

The symbol $\mathcal{I}=\{I_k\}$ denotes a sectorial open covering 
indexed by $k\in \Z/K\Z$ for some $K>1$.
Let $\Xi=(\Xi_k)_{k\in \Z/K\Z}$ be a family of isomorphisms
\[
\Xi_k\colon E\otimes \scr A_{|I_k}
\longrightarrow E^0\otimes \scr A_{|I_k},
\]
where $\scr A\subset \widetilde{\scr O}$ denotes 
the sheaf of holomorphic functions on sectors 
admitting asymptotic expansions of Gevrey order one.
The moderate part $\mathscr{A}_\mu^{\leqslant 0}$
of the coefficient sheaf 
$\scr A_\mu$ is a subsheaf of $\scr A$ 
encoding the allowed exponential growth
(The asymptotic expansion is
the constant term). 
The moderate part $\scr A_\mu^{\leqslant 0}$ controls the Stokes automorphisms, while the full sheaf $\scr{A}_\mu$ will appear in the global formulation.
We impose the following conditions on the automorphisms
$\Xi_k\circ \Xi_{k+1}^{-1}$ defined on $I_k\cap I_{k+1}$:
\begin{itemize}
    \item $\Xi_k\circ \Xi_{k+1}^{-1}$ is asymptotic to the identity on $I_k\cap I_{k+1}$;
    \item $\Xi_k\circ \Xi_{k+1}^{-1}$ is defined over the ring $\scr A_\mu^{\leqslant 0}$.
\end{itemize}
These conditions ensure that the gluing data encode 
the Stokes phenomenon and allow us to define 
the de Rham functor
for connections 
\[\nabla^k\coloneqq \Xi_k^{-1}\nabla^0\Xi_k 
\quad (k\in\Lambda_K)\]
defined on sectors $\mathcal{I}=\{I_k\}_{k\in 
\Lambda_K}$. 
See Definition \ref{dRE} for details.
We note that this generalizes the category of stalks of meromorphic connections of unramified exponential type.
See Remark \ref{dRgeneralize}. 

On the Betti side, 
we describe the Stokes data intrinsically 
in terms of filtered modules over the coefficient sheaf.
We consider the category $\Be_{C,\mu}$
whose objects are filtered $\scr A_\mu$-modules
$(\scr L,\scr L_{\leqslant })$
indexed by $\C$.
In particular, 
$\mathscr{L}$
is a sheaf of $\mathscr{A}_\mu$-modules
on $S^1$. 
The filtration reflects the exponential growth order,
and is compatible with the natural filtration 
$\scr A_\mu^{\leqslant  c}$ of the coefficient sheaf.
We require that the associated graded module
\[
\gr \scr L
=
\bigoplus_{c\in \C}
\scr L_{\leqslant  c}/\scr L_{<c}
\]
admits local lifts, 
so that  
$\scr L$ is locally isomorphic to 
$\scr{A}_\mu\otimes \gr\scr{L}$.
See Definition \ref{DefBe} for details.
This notion generalizes the category of 
Stokes-filtered local systems of unramified exponential type 
(cf. \cites{Deligne,Sabbah}),
and also recovers, in a special case,
the Stokes-filtered $\scr A_\per$-modules introduced by the author 
in the study of mild difference modules \cite{shamoto2022stokes}.
See Remark \ref{Begeneral}.

The first main theorem 
of this paper is the following:
\begin{theorem}[Theorem \ref{main theorem}]
    There is an equivalence of categories
    \begin{align*}
    \DR^\mu\colon \dR_{C,\mu}\longrightarrow\Be_{C,\mu}.
    \end{align*}
\end{theorem}
This is a generalization 
of Deligne--Malgrange Riemann--Hilbert correspondence 
for meromorphic connections 
of unramified exponential type.
Indeed, in
the case $L=0,\mu=0$, the sheaf $\scr{A}_\mu$
is the constant sheaf $\C_{S^1}$, 
$\dR_{C,\mu}$
is equivalent to the category
of meromorphic connections of 
unramified exponential type 
(Remark \ref{dRgeneralize}),
and 
$\Be_{C,\mu}$
is equivalent to 
the category of 
Stokes-filtered local systems
(Remark \ref{Begeneral}).
Under those equivalences, the functor $\DR^\mu$
corresponds to the classical de Rham functor 
of Deligne--Malgrange. 

We also have the duality 
functor $\mathbb{D}_z\colon \dR_{C,\mu}\to \dR_{C,\mu}^{\rm op}$
and hence a solution functor 
$\Sol_\mu=\DR^\mu\circ \mathbb{D}_z$.

\subsection{Comparison of isomorphisms 
in dimension one}
\label{Intro 1dim}
The Riemann–Hilbert framework introduced above
admits a concrete realization when $\dim X=1$.
For a pair $(X,\alpha)$ of dimension one
satisfying the assumptions stated below, 
we construct canonically
\begin{itemize}
    \item an object $(E_{X,\alpha},\mathcal{I},\Xi,(E_{X,\alpha}^0,\nabla^0))$ of $\dR_{C_f,\mu_\alpha}$, and
    \item an object $(\scr{L}_{X,\alpha},\scr{L}_{X,\alpha\leqslant })$ of $\Be_{C_f,\mu_\alpha}$,
\end{itemize}
where the finite set $C_f\subset \C$
is fixed in the construction 
and the map $\mu_\alpha\colon L\to \C$
is defined in the same way as in \eqref{intromu}.
The construction is parallel in spirit
to that of Kontsevich–Soibelman,
although our formulation uses
Stokes structures instead of
wall-crossing structures.
For technical reasons,
and in order to simplify the geometric description,
we restrict ourselves to the case $\dim X=1$.

The second main theorem of this paper
establishes that these objects
are naturally identified
under the functor $\Sol_{\mu_\alpha}$.
Thus, in dimension one,
the comparison between the de Rham
and Betti sides is formulated
as the compatibility
with the solution functor.
We shall explain the formulation 
briefly. 

Let $X$ be a compact Riemann surface and $\alpha$ be a meromorphic one form on $X$
which 
has a non-empty finite 
set $Z$ of zeros and a
non-empty finite set $D$ of poles.
 We also assume the support property
 for reduced period map 
 $\mu_\alpha\colon L\to \C$ defined
 by \eqref{intromu}. 
We fix a function 
$f\colon \widetilde{Y}\to \C$
over the universal covering map
$\pi_Y\colon \widetilde{Y}\to Y$
such that $df=\pi_Y^*\alpha$. 
We may fix $C_f\subset \C$
so that $C_f+\mu_\alpha(H_1(Y,\Z))$
is the set of critical values of $f$.
We set $\mu=\mu_\alpha$ in the following. 

On de Rham side,
we consider  the first cohomology of 
the $\alpha$-twisted de Rham complex
\eqref{Intro dR global},
which defines an $\C\conv{z}$-vector 
space $E_{X,\alpha}$, which is regarded 
as the global de Rham structure. 
The local de Rham structure $E^{0}_{X,\alpha}$
is localization of 
the first hypercohomology of 
\eqref{Intro localdR}.
We have a decomposition $E^0_{X,\alpha}=\bigoplus_{c\in C_f} E_c^0$.
The connection $\nabla^0$ is defined
so that $\nabla^0_{|E^0_c}-d(c/z)$
is regular singular. 
The isomorphism $\phi_\dR$
corresponds to a formal isomorphism
\[\widehat{\Xi}_{X,\alpha}
\colon E_{X,\alpha}\otimes \C\jump{z}
\longrightarrow E_{X,\alpha}^0\otimes \C\jump{z}.\] 
In the proof of the second 
main theorem,
we construct the data  
$(\mathcal{I},\Xi_{X,\alpha})$ of analytic lifts 
of $\widehat{\Xi}_{X,\alpha}$. 

On Betti side,
we consider the rapid decay homology  of 
a local system
\begin{align*}
    \scr{A}_\mu(I) \exp(-z^{-1}f)
\end{align*}
of $\scr{A}_\mu(I)$-modules
on $Y$ for each open interval
$I\subset S^1$. 
We regard the 
induced sheaf of $\scr{A}_\mu$-modules
$\scr{L}_{X,\alpha}$
as global Betti structure. 
The local Betti structure
is the first homology  
version of Kontsevich--Soibelman
\eqref{Intro Betti local},
which defines a local system $\scr{F}_{X,\alpha}$
with graded Stokes filtration with exponential factor $C_f$.
The tensor product 
$\scr{G}_{X,\alpha}=\scr{A}_\mu\otimes \scr{F}_{X,\alpha}$
is regarded as the local Betti structure. 
We have a natural morphism 
\[\eta_d^{X,\alpha}\colon \scr{G}_{X,\alpha|{\mathbb{I}_d}}\to \scr{L}_{X,\alpha|{\mathbb{I}_d}}\]
for  
$d\in \R$ with 
$(C_f,\mu)$-generic condition
(see Definition \ref{Cmugeneric})
and \[{\mathbb{I}_d}=\{e^{\i\theta}\in S^1\mid 
|\theta-d|<\pi/2\}.\] 
In the proof of the main theorem,
we see that $\scr{L}_{X,\alpha}$
is equipped with a filtration
$\scr{L}_{X,\alpha\leqslant}$ compatible
with the action of $\scr{A}_\mu$. 

The second main theorem of this paper is 
the following:
\begin{theorem}[Theorem \ref{2nd main}]
    Let $(X,\alpha)$ be as above. 
    The following holds:
    \begin{itemize}
        \item The tuple $((E_{X,\alpha},\mathcal{I},\Xi_{X,\alpha}),(E^0_{X,\alpha},\nabla^0))$ is an object of 
        $\dR_{C_f,\mu_\alpha}$.
        \item The pair $(\scr{L}_{X,\alpha},\scr{L}_{X,\alpha\leqslant })$ is an object of 
        $\Be_{C_f,\mu_\alpha}$ with
        $\gr\scr{L}_{X,\alpha}\otimes\scr{A}_\mu\simeq \scr{G}_{X,\alpha}$. 
        \item There is a commutative diagram of isomorphisms
        \begin{align}\begin{split}\label{Intro Diagram}
    \xymatrix@C=7em{
\scr{G}_{X,\alpha|{\mathbb{I}_d}}\ar[r]^{{\rm rh}^{\rm local}}\ar[d]_{\eta^{X,\alpha}_d}& \mathrm{Sol}_{\mu_\alpha }(E^0_{X,\alpha},\nabla^0)_{|{\mathbb{I}_d}}
\ar[d]^{\mathbb{D}_z(\Xi_{X,\alpha})^d}
\\
    \scr{L}_{X,\alpha|{\mathbb{I}_d}}\ar[r]^{{\rm rh}^{\rm global}}
    &\mathrm{Sol}_{\mu_\alpha}
    (E_{X,\alpha},{\Xi}_{X,\alpha})_{|{\mathbb{I}_d}}
     }
\end{split}
\end{align}
of filtered modules over $\scr{A}_\mu$. 
    \end{itemize}
\end{theorem}
We would like 
to see this theorem as 
the comparison of isomorphisms theorem for $(X,\alpha)$. 
The morphisms $\rh^\loc$ and $\rh^\gl$
in \eqref{Intro Diagram}
are the counterparts
of $\RH_\loc$ and ${\RH}_\gl^{\rm formal}$
in \eqref{KS},
respectively. 
$\mathbb{D}_z(\Xi_{X,\alpha})^d$
is the dual of an analytic lift of $\widehat{\Xi}_{X,\alpha}$
over ${\mathbb{I}_d}$, the explicit description
of which constitutes the main part of the proof.

\subsection{Further directions}
The pair $(X,\alpha)$, both in dimension one and in higher dimensions,
naturally appears as a mirror Landau-Ginzburg model
of a semi-projective variety $X^\vee$
equipped with a torus action \cite{Iritani2017EquivariantMirror}.
It would be natural to 
formulate 
equivariant gamma conjecture III 
as an equivariant generalization of \cite{sanda2020analogue}
using the Riemann--Hilbert correspondence established in this paper (Theorem \ref{main theorem}). 

It is also worth mentioning that when $\alpha$ is exact, i.e., $\alpha = df$,
the Hodge theory associated with $f$ has been extensively developed in the framework of exponential, irregular, and non-commutative Hodge structures.
This theory can be viewed as providing a substantial enrichment of the integrable twistor structure naturally attached to $f$.
See, for instance,  \cites{katzarkov2008hodge,Deligne,esnault20171,sabbah2018irregular,mochizuki2018twistor,mochizuki2025rescalability}, without attempting to give a complete list.
When $\alpha$ is closed but not exact,
one is naturally led to the theory of non-integrable twistor structures \cite{mochizuki2011wild}.
Understanding the relationship between our de Rham and Betti structures associated with $\alpha$ and the corresponding non-integrable twistor structures might be an interesting direction for future research.

\subsection{Contents of the paper}
In \S \ref{Preliminary}, 
we prepare some notations 
for Gevrey order one asymptotic analysis
and de Rham functors 
for stalks of meromorphic connections. 
We also briefly review the Borel-Laplace transform,
which will be used in a concrete example
and will play a role in future applications
to equivariant Gamma conjectures.
In \S \ref{SRH},
we formulate the Riemann--Hilbert correspondence
explained in \S \ref{Intro RH}. 
The proof is given in \S \ref{PROOF}.
The comparison of isomorphisms theorem described in \S \ref{Intro 1dim}
is formulated and proved in \S \ref{Comp1}.

\subsection{Acknowledgment}
The author would like to express his sincere gratitude to Claude Sabbah for his helpful comments on an earlier version of this manuscript.
He is also grateful to Takuro Mochizuki, Takahiro Saito, and Fumihiko Sanda for valuable discussions and encouragement.
He thanks the anonymous referee for their careful reading and valuable suggestions.

This study is supported by JSPS KAKENHI Grant Numbers 24K16925 and 20K14280. 
This work was also supported by the Research Institute for Mathematical Sciences,
an International Joint Usage/Research Center located in Kyoto University.

On behalf of all authors, the corresponding author states that there is no conflict of interest.

\section{Preliminary}\label{Preliminary}
In this section, we fix some notation and terminology used in this paper on the Borel-Laplace transformation and the de Rham functor for stalks of meromorphic connections.
We also recall some standard facts. 
Since these facts are standard, we refer the reader to the textbooks \cite{Balser2000} and \cite{Sabbah}. 
We fix the notations for cyclic covers.

\subsection{Borel-Laplace transform}
\label{SBL}
Let $\C$ be the set of complex numbers. 
A formal power series 
$\widehat{f}(z)=\sum_{n=0}^\infty a_nz^n$ $(a_n\in \C)$
is called a Gevrey power series of order one,
if there exists a positive constant 
$C>0$ such that 
\begin{align*}
    |a_n|<C^{n+1}n!
\end{align*}
for all $n\geq 0$. 
Let $\C\jump{z}_1$ denote the ring of formal power 
series of Gevrey order one. 
For $d\in \R$ and $\Theta,\rho>0$,
we set \[S=S_z=S_z(d,\Theta,\rho)=\{z=r e^{\i\theta}\mid 0<r<\rho,|d-\theta|<\Theta/2\},\]
which we call a sector (in the $z$-plane).
A closed subsector $\overline{S'_z}$ in $S_z$
is a subset $\overline{S'_z}\subset S_z$
defined as 
$\overline{S_z'}=\overline{S_z'}(d',\Theta',\rho')=\{z=re^{\i\theta}\in \C\mid 
0<r\leq \rho', |\theta-d'|\leq \Theta'\}$
for some $d'\in \R$ and $\Theta',\rho'>0$.
\begin{definition}
    Let $f$ be a holomorphic function 
 on a sector $S_z$. Let $\widehat{f}(z)=\sum_{n=0}^\infty a_nz^n$ be an element in $\C\jump{z}_1$. 
 We say that $f$ has $\widehat{f}$ as its 
    asymptotic expansion (of Gevrey order one)
    if for any 
    closed subsector $\overline{S'_z}\subset S_z$
    there exists a positive constant $C>0$
    such that for any positive integer $N$,
    the inequality
    \begin{align*}
        |z|^{-N}\left|f(x)-\sum_{n=0}^{N-1}a_nz^n\right|<C^{N+1}N!
    \end{align*}
    holds for any $z\in \overline{S_z'}$. 
    The set of holomorphic functions on $S_z$
    which have an element in $\C\jump{z}_1$
    as its expansion 
    is denoted by $A(S_z)$. 
\end{definition}
The set $A(S_z)$ is known to be a subring
of the ring of holomorphic functions on $S_z$.
For $f\in A(S_z)$,
its asymptotic expansion $\widehat{f}$ is known to be unique.
Hence we use the notation 
    $\asy(f)=\widehat{f}$.
It is also known that the map 
$\asy\colon A(S_z)\to \C\jump{z}_1, f\mapsto \asy(f)$
is a ring homomorphism.

Let $\C\{\zeta\}$ denote the ring of convergent power series in the variable $\zeta$. 
The formal Borel transform
\begin{align*}
    \widehat{\cal{B}}\colon \C\jump{z}_1\longrightarrow\C\{\zeta\}
\end{align*}
is defined as \[\widehat{\cal{B}}(\widehat{f})=\sum_{n=0}^\infty (n!)^{-1}a_n\zeta^n\]
for $\widehat{f}(z)=\sum_{n=0}^\infty a_nz^n\in \C\jump{z}_1$.

For $d\in \R$ and $\varepsilon>0$,
we set 
$S_\zeta=S_\zeta(d,\varepsilon)=\{\zeta=re^{\i\theta}\in  \C\mid r>0,|\theta-d|<\varepsilon \}$,
which we call an unbounded sector in  the $\zeta$-plane. 
A closed subsector of $S$ is defined 
as a subset $\overline{S'_\zeta}\subset S_\zeta$
of the form $\overline{S'_\zeta}=\overline{S'_\zeta}(d',\varepsilon')= \{\zeta=re^{\i\theta}\in\C\mid |\theta-d'|\leq \varepsilon'\}$
for $d'\in \R$ and $\varepsilon'>0$. 
\begin{definition}
    For an unbounded sector $S$ in $\zeta$-plane and 
    $g(\zeta)\in \C\{\zeta\}$,
    we say that $g(\zeta)$ is of exponential
    size one on $S_\zeta$ if the following conditions hold:
    \begin{itemize}
        \item There exists an analytic continuation of 
    $g$ on $S_\zeta$, which is also denoted by $g$.
        \item For any closed subsector $\overline{S'_\zeta}\subset  S_\zeta$, there exist constants
        $C,h>0$ such that 
        the inequality
        $|g(\zeta)|\leqslant Ce^{h|\zeta|}$
        holds for $\zeta\in \overline{S'_\zeta}$. 
    \end{itemize}
\end{definition}
For $g(\zeta)=\sum_{m=0}^\infty b_m\zeta^m\in\C\{\zeta\}$ which is of
exponential size one 
on $S_\zeta(d,\varepsilon)$ for some $\varepsilon>0$,
we define the Laplace transform in
the direction $d$ by
\begin{align*}
    \cal{L}_d(g)(z)=z^{-1}\int_0^{\infty e^{\i d}}g(\zeta)e^{-\zeta/z}d\zeta.
\end{align*}
Here, the integral is defined 
over the path $\R_{\geq 0}\to\C, t\mapsto e^{\i d}t$. 
It is known that for any $\epsilon>0$
there exists $\rho(\epsilon)>0$ such that 
\[\cal{L}_d(g)(z)\in A(S_z(\pi+\varepsilon-\epsilon ,\rho(\epsilon))).\] 
\begin{theorem}[\cite{Balser2000}*{Theorem 33}]\label{summable}
    For $\widehat{f}\in \C\jump{z}_1$,
    and a direction $d\in\R$, 
    the following conditions are equivalent to each other$:$
    \begin{itemize}
        \item There exist $\Theta>\pi$, $\rho>0$, and  $f\in {A}(S_z(d,\Theta,\rho))$ such that 
         $\asy(f)=\widehat{f}$. 
        \item $\widehat{\cal{B}}(\widehat{f})(\zeta)$ is 
        of exponential size one on 
        $S_\zeta(d;\varepsilon)$
        for some  $\varepsilon>0$.
    \end{itemize}
    Moreover, if these conditions are satisfied, then we have $f=\cal{L}_d(\widehat{\cal{B}}_d(\widehat{f}))$.
\end{theorem}
A formal series
$f\in\C\jump{z}_1$
is called $1$-summable along $d$ if it satisfies the conditions in 
Theorem \ref{summable}. 
We set 
\begin{align*}
    \C\{z\}_{1,d}=\{f\in\C\jump{z}_1\mid f\text{ is $1$-summable along $d$.}\},
\end{align*}
which is known to be a subring in $\C\jump{z}_1$.
We then set $\cal{S}_d=\cal{L}_d\circ \widehat{\cal{B}}$. 
We note that for $f,g\in \C\{z\}_{1,d}$,
equations $\cal{S}_d(fg)=\cal{S}_d(f)\cal{S}_d(g)$
and $\cal{S}_d(f+g)=\cal{S}_d(f)+\cal{S}_d(g)$
hold.
Hence $\cal{S}_d$
is extended to 
the quotient field of $\C\{z\}_{1,d}$.

\subsection{Sheaves on a circle} \label{sheaves on S^1}
Let ${S}^1=\{w\in\C\mid |w|=1\}$
be the unit circle. 
Let $\widetilde{\C}_0=\{(z,w)\in \C\times {S}^1\mid z=|z|w\}$
be the real oriented blow-up 
of $\C$ at the origin. 
Let $\varpi\colon\widetilde{\C}_0\to \C$,
defined by
$\varpi(z,w)=z$,
denote the projection. 
Set $\C^*=\C\setminus \{0\}$.
Let $\jmath\colon \C^*\to \widetilde{\C}_0$,
$\jmath(z)=(z,z/|z|)$ and 
$\imath\colon {S}^1\to \widetilde{\C}_0$,
$\imath(w)=(0,w)$ denote the inclusions. 
Via the inclusion $\imath$,
we regard $S^1$ as the boundary of $\widetilde{\C}_0$. 

For a complex manifold 
$M$ such as $\C$ or $\C^*$,
let $\scr{O}_M$ denote the sheaf of 
rings of holomorphic functions on $M$. 
We then set $\widetilde{\scr{O}}=\imath^{-1}\jmath_*\scr{O}_{\C^*}$, which is a sheaf of rings on $S^1$.
Let  
$\scr{A}^{\leqslant 0}$ be the subsheaf of moderate growth
functions in 
$\widetilde{\scr{O}}$. 
Here, 
    $f\in \widetilde{\mathscr{O}}$ on $I\subset S^1$, represented by a holomorphic function 
    $\widetilde{f}\in \widetilde{\jmath}_*\scr{O}_{\C^*}$ on an open subset $\widetilde{I}\subset\widetilde{\C}_0$ with 
    $I=S^1\cap \widetilde{I}$ is \textit{of moderate growth}
    if for any compact subset $K\subset \widetilde{I}$, there exist constants $C_K>0$
    and $N_K\geq 0$ such that 
    \[ |f(z) |\leq C_K|z|^{-N_K} \text{ for any }
    z\in K\setminus I.
    \]

We also define the sheaf
$\scr{A}\subset \widetilde{\scr{O}}$
of functions with asymptotic expansions 
of Gevrey order one. 
For connected $I\subset S^1$,
using the notation as above,
a section $f\in \scr{O}(I)$
is in $\scr{A}(I)$ if there is
$\widetilde{I}$ and 
$\widetilde{f}$ as above 
with the condition that 
we have $\widetilde{f}\in A(S_z)$
for any sector $S_z$ in $z$-plane 
such that $\jmath(S_z)\subset \widetilde{I}$. 
By Borel-Ritt theorem,
we have the exact sequence 
\begin{align*}
    0\longrightarrow \scr{A}^{<0}
    \longrightarrow \scr{A}\xrightarrow{\asy}
    \C\jump{z}_{1,S^1}\longrightarrow 0,
\end{align*}
where $\C\jump{z}_{1,S^1}$ denotes the constant 
sheaf, $\asy$ is the asymptotic expansion 
defined above, and $\scr{A}^{<0}$
is defined as the kernel of $\asy$. 
We note that $\scr{A}^{<0}\subset\scr{A}\subset\scr{A}^{\leqslant 0}$. 
\begin{remark}
    In the usual notation,
    $\scr{A}$ and $\scr{A}^{<0}$ are denoted by
    $\scr{A}_1$ and $\scr{A}^{<0}_1$,
    respectively. 
    However, we will not use the sheaf
    $\scr{A}_k$ of functions which have 
    asymptotic expansion of Gevrey order $k$  for $k\neq 1$ in this paper. 
    Hence we omitted the 
    subscript $1$. 
\end{remark}
For $c\in \C$,
$\exp(c/z)$ can be regarded as 
a global section of $\widetilde{\scr{O}}$.
We then set
$\scr{A}^{\leqslant c}=\exp(c/z)\scr{A}^{\leqslant 0}$ 
and 
$\scr{A}^{<c} =\exp(c/z)\scr{A}^{<0}$.
We have $\scr{A}^{\leqslant c}\cdot \scr{A}^{\leqslant c'}\subset \scr{A}^{\leqslant c+c'}$ and $\scr{A}^{\leqslant c}\cdot \scr{A}^{<c'}\subset \scr{A}^{< c+c'}$
for any $c,c'\in \C$. 
For $c,c'\in\C$
and $\theta\in \R$,
we define 
\begin{align*}
    c<_\theta  c'
    \overset{\rm def}{\Longleftrightarrow}
    \mathrm{Re}(e^{-\i\theta}(c-c'))<0. 
\end{align*}
We then define $c\leqslant_\theta c'\overset{\rm def}{\Longleftrightarrow} c<_\theta c$ or $c=c'$.
For an open subset $I\subset S^1$, we say
$c\leqslant_Ic'$ holds
iff we have $c\leqslant_\theta c'$ for any $\theta$ with $e^{\i\theta}\in I$. 
We note that if $c<_Ic'$,
then we have $ \scr{A}^{\leqslant c}_{|I}\subset \scr{A}_{|I}^{<c'}$. 
We set $\scr{A}^{\gr_c}=\scr{A}^{\leqslant c}/\scr{A}^{<c}$ for $c\in \C$.

\subsection{De Rham functors for meromorphic connections}
Let $\scr{O}=\C\{z\}$ denote the ring of convergent power 
series. Let $\C\conv{z}$ be the fractional field of $\scr{O}$.

A stalk of a meromorphic function over ${{\C\conv{z}}}$ 
is a pair $(E,\nabla)$
of 
 finite dimensional ${{\C\conv{z}}}$-vector space $E$ 
and a $\C$-linear map $\nabla\colon E\to Edz$
such that the equality
\begin{align}\label{Lei}
\nabla(fv)=vdf+f\nabla v
\end{align}
holds for any $f\in {\C\conv{z}}$ and $v\in E$.
Here, $d\colon {\C\conv{z}}\to {\C\conv{z}}dz$ denotes
the exterior derivative. 
The $\C$-linear map $\nabla$
is called the connection on $E$. 

For two stalks $(E,\nabla)$
and $(F,\nabla)$ of meromorphic connections over ${\C\conv{z}}$,
the ${\C\conv{z}}$-vector space $\Hom_{{\C\conv{z}}}(E,F)$
of ${\C\conv{z}}$-linear maps from $E$ to $F$
is equipped with the connection $\nabla$.
We use the notation 
\[\hom(E,F)=(\Hom_{{\C\conv{z}}}(E,F),\nabla)\] in this paper. We also set $\End(E) =\hom(E,E)$.
The $\C$-vector space of vectors $f\in\Hom_{{\C\conv{z}}}(E,F)$ with $\nabla(f)=0$
is denoted by $\Hom(E,F)^{\nabla}$.
 
\begin{definition}
    The category 
$\mathsf{Mero}_{{\C\conv{z}}}$ of stalks of meromorphic connections over ${{\C\conv{z}}}$
is defined as follows:
\begin{itemize}
    \item An object of $\mathsf{Mero}_{{\C\conv{z}}}$
    is a stalk of meromorphic connection over ${{\C\conv{z}}}$.
    \item The set of morphisms in $\mathsf{Mero}_{{\C\conv{z}}}$
    between two objects $(E,\nabla)$
and $(F,\nabla)$
is $\Hom(E,F)^{\nabla}$. 
The composition is defined as the composition 
of maps. 
\end{itemize}
\end{definition}

For $(E,\nabla)\in \mathsf{Mero}_{{\C\conv{z}}}$,
we consider de Rham complexes $\DR(E)=[E\xrightarrow{\nabla}Edz]$
and 
\[\widetilde{\DR}(E)=[E_{S^1}\otimes_{{\C\conv{z}}} \widetilde{\mathscr{O}}\xrightarrow{\widetilde\nabla}E_{S^1}\otimes_{{\C\conv{z}}} \widetilde{\scr{O}}dz]
\]
concentrated in degrees $0$ and $1$
where we put $\widetilde{\nabla}=\nabla\otimes\id+\id \otimes d$. The subscripts $S^1$ and ${{\C\conv{z}}}$ will often be 
omitted.

We also define 
$\DR_{\leqslant c}(E,\nabla)$ 
and $\DR_{< c}(E,\nabla)$ for $c\in\C$
as follows:
\begin{align*}
    \DR_{\leqslant c}(E,\nabla)=[E\otimes\scr{A}^{\leqslant c}\xrightarrow{\widetilde{\nabla}} E\otimes \scr{A}^{\leqslant c}dz],\\
    \DR_{< c}(E,\nabla)=[E\otimes\scr{A}^{< c}\xrightarrow{\widetilde{\nabla}} E\otimes \scr{A}^{< c}dz]\\
    \DR_{\gr_c}(E,\nabla)=[E\otimes\scr{A}^{\gr_c}
    \xrightarrow{\widetilde{\nabla}} E\otimes \scr{A}^{\gr_c}dz]
\end{align*}
\begin{remark}
    For an open subset $I\subset S^1$
    and $\scr{A}_{|I}$-module $\scr{M}$
    with connection $\nabla\colon \scr{M}\to \scr{M}dz$ with Leibnitz rule \eqref{Lei}
    for $f\in\scr{A}_{|I}$ and $v\in\scr{M}$,
    we define 
    the complexes $\widetilde{\DR}(\scr{M})$,
    $\DR_{\leqslant c}(\scr{M})$,
    $\DR_{<c}(\scr{M})$ and $\DR_{\gr_c}(\scr{M})$ in 
    a similar way. 
\end{remark}
Although the definition 
of $\DR_{< c}(E,\nabla)$
is different from that in \cite{Sabbah}, 
we obtain the same complex for 
the following class of 
objects:
\begin{definition}
    An object $(E^0,\nabla^0)\in\mathsf{Mero}_{{\C\conv{z}}}$
    is called elementary exponential if it is isomorphic to 
    the direct sum $\bigoplus_{c\in \C}({{\C\conv{z}}}^{r_c},\nabla^{0,c})$
    of the form
    \begin{align*}
        \nabla^{0,c}=d-(c\cdot \id_{{{\C\conv{z}}}^{r_c}}-zA_c)\frac{dz}{z^2}
    \end{align*}
    for $r_c=0,1,\dots$ and $A_c\in \End_\C(\C^{r_c})$. 
    The finite set $\{c\in\C\mid r_c> 0\}$
    is called the exponential factor of $E^0$. 
    
    An object $(E,\nabla)\in \mathsf{Mero}_{{\C\conv{z}}}$
    is called unramified of exponential type if 
    there is an elementary exponential connection $(E^0,\nabla^0)$
    such that we have an isomorphism 
    \[\widehat{\Xi}\colon E\otimes_{\scr{O}}\C\jump{z}_1\to E^0\otimes_\scr{O} \C\jump{z}_1\]
    with $\widehat{\Xi}\circ \widehat{\nabla}=\widehat{\nabla}^0\circ\widehat{\Xi}$, where  
    $\widehat{\nabla}=\nabla\otimes\id_{\C \jump{z}_1} +\id_{{\C\conv{z}}}\otimes d$, and $\widehat{\nabla}^0=\nabla^0\otimes\id +\id\otimes d$.  
\end{definition}
Let $k$ be an integer. 
For a complex $C^\bullet$ of abelian groups, 
$H^k(C^\bullet)$ denotes the $k$-th cohomology
groups. 
For a complex $\scr{C}^\bullet$
of sheaves of modules on a topological space $M$, $\scr{H}^k(\scr{C}^\bullet)$ denotes 
the $k$-th cohomology sheaves
and ${H}^k(M,\scr{C}^\bullet)$
denotes the $k$-th hypercohomology group.
For a sheaf $\scr{G}$ of  non-commutative 
groups on $M$, $H^1(M,\scr{G})$
denotes the first cohomology set of $\scr{G}$.

\subsection{Cyclic covering}
For an integer $K>0$,
we set $\Lambda_K=\Z/K\Z$.
A \textit{cyclic covering} of $S^1$
is an open covering $\mathcal{I}=(I_k)_{k\in\Lambda_K}$, ${S}^1=\bigcup_{k\in \Lambda_K}I_k$ 
by open arcs $I_k$ such that 
$I_{k,\ell}\coloneqq I_k\cap I_\ell$ is non-empty if 
and only if $\ell\in \{k-1,k,k+1\}$.
Furthermore, we require that the non-empty intersection $I_k\cap I_\ell$ consists of 
at most two connected components. 
Note that $I_0\cap I_1$ consists of 
two connected components when $K=2$. 

A \textit{cyclic refinement} of a 
cyclic cover $\mathcal{I}=(I_k)_{k\in\Lambda_K}$ is a pair 
$(\cal{I}',\mathsf{u})$, comprising another cyclic cover $\cal{I}'=(I'_\ell)_{\ell\in\Lambda_L}$ 
and a map $\mathsf{u}\colon \Lambda_L\to\Lambda_K$
such that \[I'_{\ell}\subset I_{\mathsf{u}(\ell)}\]
for every $\ell\in\Lambda_L$.
We use the notation 
$\cal{I}'
\prec_{\mathsf{u}}\mathcal{I}$
to express that $(\cal{I}',\mathsf{u})$
is a refinement of $\mathcal{I}$.
For two cyclic coverings 
$\mathcal{I}$ and $\cal{I}'$,
there exists a covering 
$\cal{I}''$ such that 
$\cal{I}''\prec_{\mathsf{u}} \mathcal{I}$
and $\cal{I}''\prec_{\mathsf{v}}\cal{I}'$. 
We call the triple $(\cal{I}'',\mathsf{u},\mathsf{v})$
a \textit{mutual refinement}
of $\mathcal{I}$ and $\cal{I}'$.

\section{Riemann--Hilbert correspondence}
\label{SRH}
Let $L$ be a lattice, i.e., an 
abelian group which is isomorphic to 
$\Z^{\mathrm{rank}(L)}$ for some non-negative integer $\mathrm{rank}(L)$.
Let $\mu\colon L\to \C$ be a morphism of additive groups. We fix a norm $\|\cdot\|$ on 
$L\otimes_\Z \R$ 
and assume
that $\mu$ satisfies the following 
support property
\cite{kontsevich2008stability}: there exists a constant $R>0$ such that 
\begin{align}\label{support property}
    R\|\gamma\|\leq  |\mu(\gamma)|\quad(\forall \gamma\in L).
\end{align}
In particular, $\mu$ is assumed to be
injective. 

Let $C$ be a finite subset of $\C$.
We assume that 
for distinct $c,c'\in C$,
there is no $\gamma\in L$
with $c=c'+\mu(\gamma)$.
In other words, we assume that 
each element in $C$ represents 
a distinct $L$-orbit 
in $C+\mu(L)\subset \C$. 

In this section,
we formulate a Riemann--Hilbert correspondence 
for each $(C,\mu)$. 
We set
\begin{align*}
    \Omega_{C,\mu}\coloneqq\{(c-c')+\mu(\gamma)\mid 
    c,c'\in C,\gamma\in L\}\setminus\{0\}.
\end{align*}
\begin{definition}\label{Cmugeneric}
    A direction $d\in\mathbb{R}$
    is called $(C,\mu)$-generic
    if the half line
    \[\ell_d=\{\zeta=re^{\i d}\mid 
    r> 0 \}\]
    does not intersect $\Omega_{C,\mu}$. 
\end{definition}
Since $\Omega_{C,\mu}$
is a countable set, 
the set of non-generic directions is countable.
Hence the set of generic directions is dense in $\R$. 
When $\mathrm{rank}(L)>1$,
the set of non-generic directions is also dense by \eqref{support property}.

\subsection{A filtered sheaf of rings}\label{Amu}
For an open interval $I\subset{S}^1$,
and $\gamma\in L$ with $\mu(\gamma)<_I0$,
we have $\exp(\mu(\gamma)/z)
\in \scr{A}^{<0}(I)$.
Generalizing this observation,
we have the following:
\begin{lemma}
A series
\begin{align*}
f(z)=
\sum_{\mu(\gamma)\leqslant_I0}
a_\gamma\exp(\mu(\gamma)/z),
\quad (a_\gamma\in \C)
\end{align*}
is in $\scr{A}(I)$
and $f(z)-a_0\in\scr{A}^{<0}(I)$
if we have 
\begin{align*}
\|f\|_{I,\varrho}
\coloneqq \sum_{\mu(\gamma)\leqslant_I0 }|a_\gamma|\varrho^{\|\gamma\|}<\infty
\end{align*}
for some $\varrho>0$.
\end{lemma}
\begin{proof}
Assume that $\|f\|_{I,\varrho}$ is finite. 
By the condition, we have 
$\mu(\gamma)<_I0$ for $a_\gamma\neq 0$, $\gamma \neq 0$. 
For any compact subset $J\subset I$,
there exists a constant $\varepsilon_J>0$
such that 
we have $\cos(\arg(\mu(\gamma))-\theta)<-\varepsilon_J$ $(\theta\in J)$ 
for any such $\gamma$. 
Then, 
by the support property \eqref{support property}, 
we have 
\begin{align}\label{inequality for A}
\begin{split}
    |f(z)-a_0|
    &\leq \sum_{\mu(\gamma)<_I0}|a_\gamma|
    \exp(\mathrm{Re}(\mu(\gamma)/z))\\
    &\leq \sum_{\mu(\gamma)<_I0}|a_\gamma|
    \exp(-\varepsilon_JR\|\gamma \|/|z|)
\end{split}
\end{align}
for any $z$ with $\arg(z)\in J$. 
This inequality implies the desired result (see Appendix \ref{Appendix} for a more precise estimate). 
\end{proof}

\begin{definition}
    For a connected open 
    $I\subset S^1$ and $\varrho>0$, we set 
    \begin{align*}
        \scr{A}_{\mu,\varrho}^{\leqslant 0}(I)\coloneqq 
        \left\{f(z)=\sum_{\mu(\gamma)\leqslant _I0}a_\gamma \exp(\mu(\gamma)/z)\middle| 
        \|f\|_{I,\varrho}<\infty
        \right\}.
    \end{align*}
    The sheaf associated to the presheaf $I\mapsto \scr{A}_{\mu,\varrho}^{\leqslant 0}(I)$ is denoted by $\scr{A}_{\mu,\varrho}^{\leqslant 0}$.
    We then set  
    $\scr{A}_{\mu}^{\leqslant 0}(I)\coloneqq \bigcup_{\varrho>0} \scr{A}_{\mu,\varrho}^{\leqslant 0}(I)$
    and the associated sheaf is denoted by 
    $\scr{A}_\mu^{\leqslant 0}$.
\end{definition}

We also set
\[\scr{A}_\mu=\sum_{\gamma\in L}
\exp(\mu(\gamma)/z)\scr{A}_\mu^{\leqslant 0}.\]
We then obtain a filtration 
$\scr{A}_\mu^{\leqslant c}=\scr{A}_\mu\cap \scr{A}^{\leqslant c}$
and $\scr{A}_\mu^{< c}=\scr{A}_\mu\cap \scr{A}^{ <c}$
for $c\in \C$. 
In the following, we regard $\scr{A}_\mu$ as a sheaf of filtered rings indexed by $\C$ equipped with this filtration.

 \begin{lemma}\label{vanishing A_mu}
     We have $H^1(I,\scr{A}_\mu^{<0})=H^1(I,\scr{A}_\mu^{\leqslant0})=0$
     if $I$ is contained in  
     \[{\mathbb{I}_d}\coloneqq \{e^{\i\theta}\mid |\theta-d|<\pi/2\}\]for some $d\in \R$. 
 \end{lemma}
 \begin{proof}
By the Leray-Grothendieck theorem, we can use \v{C}ech cohomology.
     Consider a covering
     $I=I_1\cup I_2$ of an open interval 
     $I$ contained in ${\mathbb{I}_d}$ for some $d \in \R$
     by open intervals $I_j$ ($j=1,2$). 
     Set $I_0=I_1\cap I_2$. 
     If $\mu(\gamma)<_{I_0}0$,
     then either $\mu(\gamma)<_{I_1}0$
     or $\mu(\gamma)<_{I_2}0$ holds. 
     It follows that the map 
     \begin{align*}
         H^0(I_1,\scr{A}_\mu^{< 0})\oplus 
         H^0(I_2,\scr{A}_\mu^{<0})\longrightarrow H^0(I_0,\scr{A}_\mu^{<0}),\quad 
         (f_1,f_2)\mapsto f_{2|I_0}-f_{1|I_0}
     \end{align*}
     is surjective. 
     This implies the desired result 
     by the standard argument 
     in computing \v{C}ech cohomology. 
 \end{proof}

\subsection{De Rham category}

Let $(E^0,\nabla^0)$
be an elementary exponential meromorphic connection. 
Let $\mathscr{E}nd^{<0}_\mu(E^0)$ denote the 
subsheaf of $\End(E^0)\otimes \scr{A}$
defined as follows: 
\begin{align*}
    \scr{E}nd_\mu^{<0}(E^0)\coloneqq \sum_{c+c'<0}
    \scr{A}_\mu ^{\leqslant c}\scr{H}^0\DR_{\leqslant c'}(\End(E^0)).
\end{align*}
The inequality $c+c'<0$ means that the stalk of $\scr{E}nd_\mu^{<0}(E^0)$ at a point $e^{\i\theta}$ is defined as the sum over $c,c'$ satisfying $c+c'<_\theta 0$.
We then set 
\[\aut^{<0}_\mu(E^0)\coloneq \id_{E^0}+\scr{E}nd_\mu^{<0}(E^0),\]
which is a sheaf of groups. 

\begin{definition}\label{dRE}
    We define the category ${\mathsf{dR}}_{C,\mu}$
    as follows: 
    \begin{itemize}
        \item An object in $\mathsf{dR}_{C,\mu}$
    is a tuple $E=((E,\mathcal{I},\Xi),(E^0,\nabla^0))$
    consisting of 
    \begin{itemize}
    \item an elementary exponential connection 
    $(E^0,\nabla^0)$ whose exponential factors are contained in $C$.
        \item a finite dimensional ${\C\conv{z}}$-vector space
    $E$,
    \item  a  cyclic covering
    $\mathcal{I}=(I_k)_{k\in\Lambda_K}$
    of $S^1$, and 
    \item a family $\Xi=(\Xi_k)_{k\in\Lambda_K}$ of isomorphisms
    \[\Xi_k\colon E\otimes\scr{A}_{|I_k}\longrightarrow
    E^0\otimes \scr{A}_{|I_k}\]
    such that  
    \begin{align}\label{aut}
    \Xi_k\circ \Xi_{k+1}^{-1}
    \in H^0(I_{k,k+1},\aut_\mu^{<0}(E^0))
    \end{align}
    for any $k\in \Lambda_K$. 
    \end{itemize}       
        \item  The set of morphisms from $E=(E,\mathcal{I},\Xi)$
    to $F=(F,\cal{I}',\Theta)$ denoted by 
    \[\Hom_{\dR}(E,F)\]
    is the set of pairs
    $f=(f,f_0)$
    of 
    \begin{itemize}
        \item a ${\C\conv{z}}$-linear map 
        $f\colon E\to F$, and 
        \item a flat morphism 
        $f^0\colon (E^0,\nabla^{E,0})\to (F^0,\nabla^{F,0})$ with the following property:
        There exists a 
        mutual refinement  $(\cal{I}''=(I''_p)_{p\in\Lambda_P},\mathsf{u},\mathsf{v})$
        of $\mathcal{I}$ and $\cal{I}'$
        such that \begin{align}\label{mor}
        \Theta_{\mathsf{v}(p)}\circ f_{}\circ \Xi_{\mathsf{u}(p)}^{-1}
        =f^0_{}
        \end{align}
        for any $p\in\Lambda_P$,
        where we set $\Theta_{\mathsf{v}(p)}=\Theta_{\mathsf{v}(p)|I''_p}$
        and $\Xi_{\mathsf{u}(p)}=\Xi_{\mathsf{u}(p)|I''_p}$. 
    \end{itemize}
    \item For two morphisms $f\in \Hom_{\dR}(E,F)$ and $g\in \Hom_\dR(F,G)$, the composition $g \circ f$ is defined by the pair $(g\circ f, g^0\circ f^0)$, which defines an element in $\Hom_\dR(E,G)$.
    \end{itemize}
\end{definition}
\begin{remark}
    By condition \eqref{aut},
    the asymptotic expansions of $\Xi_k$
    coincide:
    $\widehat{\Xi}_k=\widehat{\Xi}_{k+1}$
    for all $k\in\Lambda_K$.
    We call $\widehat{\Xi}$
    the \textit{underlying formal isomorphism} of $E$.
    We also call the connection \[\widehat{\nabla}\coloneqq \widehat{\Xi}^{-1}\circ \widehat{\nabla}^0\circ \widehat{\Xi}\]
    \textit{the underlying formal connection} of $E$. 
    If $(\cal{I}',\mathsf{u})$
    is a cyclic refinement of $\mathcal{I}$,
    then $(E,\cal{I}',\Xi_{\cal{I}'})$
    with $\Xi_{\cal{I}'}=(\Xi_{\mathsf{u}(\ell)|I'_\ell})_{\ell\in {\Lambda}_L}$
    is isomorphic to $(E,\mathcal{I},\Xi)$. 
\end{remark}

\begin{definition}
    For an object \[((E,\mathcal{I},\Xi),(E^0,\nabla^0))\in\mathsf{dR}_{C,\mu}\]
    with notations in Definition \ref{dRE} and $k\in \Lambda_K$,
    we set 
    \[\nabla^k\coloneqq \Xi_k^{-1}\nabla^0\Xi_k.\] 
\end{definition}

\begin{remark}\label{dRgeneralize}
    When $L=0$ and $\mu=0$,
    we have \[\Xi_k\circ\Xi_{k+1}^{-1}\in H^0(I_{k,k+1},\aut^{<0}(\End(E^0)))\] 
    where $\aut^{<0}(E^0)=\id+\DR_{<0}(\End(E^0))$. 
    Then, we have $\nabla^k=\nabla^{k+1}$.
    Hence we obtain a meromorphic connection 
    $\nabla$ of 
    exponential type on $E$.
    Moreover, in this case, the category $\dR_{C,\mu}$
    is equivalent to the category of meromorphic connections of exponential type whose exponential 
    factor is contained in $C$. 
    In this sense, the category $\mathsf{dR}_{C,\mu}$ generalizes the category of meromorphic connections of unramified exponential type.
\end{remark}

\begin{theorem}
    The category $\dR_{C,\mu}$
    is an abelian category. \qed
\end{theorem}

\subsection{Betti category}
    Let $\scr{L}=(\scr{L},\scr{L}_\leqslant)$
    be a filtered module over the
    sheaf of filtered rings $\scr{A}_\mu$
    indexed by $\C$.
    In other words,
    $\scr{L}$ is an $\scr{A}_\mu$-module
    and
    \[\scr{L}_{\leqslant }=\{\scr{L}_{\leqslant c}\}_{c\in \C}\]
    is a family of $\scr{A}_\mu^{\leqslant 0}$-submodules of $\scr{L}$
    indexed by $c\in \C$
    satisfying the following properties:
    \begin{itemize}
        \item[(a)] If $c\leqslant_\theta c'$,
        then $(\scr{L}_{\leqslant c})_{e^{\i\theta}}\subset (\scr{L}_{\leqslant c'})_{e^{\i\theta}}$.
        \item[(b)] For any $\gamma\in L$ and
        $c\in \C$,
        we have \[u^\gamma \scr{L}_{\leqslant c}=\scr{L}_{\leqslant c+\mu(\gamma)},\]
        where we set $u^\gamma=\exp(\mu(\gamma)/z)$.
    \end{itemize}
    For any $c\in\C$, we define the
    subsheaf $\scr{L}_{<c}\subset \scr{L}_{\leqslant c}$
    by the condition on stalks:
    \begin{align*}
        (\scr{L}_{<c})_{e^{\i\theta}}
        =\sum_{c'<_\theta c}(\scr{L}_{\leqslant c'})_{e^{\i\theta}}.
    \end{align*}
    We set $\gr_c\scr{L}\coloneqq \scr{L}_{\leqslant c}/\scr{L}_{<c}$ for any $c\in \C$,
    and set $\gr\scr{L}\coloneqq \bigoplus_{c\in \C}\gr_c\scr{L}$.
    By condition (b) above, the module
    $\gr\scr{L}$
    is naturally equipped with a $\C[L]$-module structure
    induced by the action of $u^\gamma=\exp(\mu(\gamma)/z)$ $(\gamma \in L)$.
\begin{definition}\label{DefBe}
    Let $\mathsf{Be}_{C,\mu}$
    be a category defined as follows:
    \begin{itemize}
        \item an object in $\Be_{C,\mu}$
        is a filtered $\scr{A}_\mu$-module
        $\scr{L}=(\scr{L},\scr{L}_{\leqslant})$ indexed by $\C$ 
        with the following properties: 
        \begin{itemize}
            \item[(a)] For each $c\in\C$, the sheaf $\gr_c\scr{L}$ is a local system of 
                  $\C$-vector spaces. 
            \item[(b)] The set $\{c\in \C\mid \gr_c\scr{L}\neq 0\}$ is contained in 
            $-C+\mu(L)$. 
            \item[(c)] For each point $e^{\i\theta}\in S^1$, there exists an open neighborhood $I$ of $e^{\i\theta}$ and a filtered isomorphism
            \begin{align}\label{splitting}
                \eta_I\colon {\scr{A}_{\mu}}_{|I}\otimes_{\C[L]}\gr\scr{L}_{|I}
                \simeqto\scr{L}_{|I}
            \end{align}
            satisfying $\gr(\eta_I)=\id$. 
        \end{itemize}
        \item a morphism between two objects
        is a morphism of $\scr{A}_\mu$-modules
        which preserves the filtration. 
    \end{itemize}
    An object $\scr{G}\in\mathsf{Be}_{C,\mu}$ is called \textit{graded} if it is isomorphic 
    to $\scr{A}_\mu\otimes_{\C[L]}\gr\scr{G}$
    as filtered $\scr{A}_\mu$-modules. 
\end{definition}

\begin{remark}\label{Begeneral}
    When $L=0$, 
    $\mathsf{Be}_{C,\mu}$
    is the category of 
    Stokes structures of exponential 
    meromorphic connections whose exponential factors are contained in $C$. 
    On the other hand,
    when $L=\Z$ and $\mu(n)=2\pi\i n$,
    this category is the same 
    as the category of Stokes structures 
    for unramified mild difference modules
    \cite{shamoto2022stokes}
    whose exponential factors are contained in $\{cs\mid c\in C+\mu(L)\}$
    if we set $s=1/z$.  
\end{remark}

\begin{theorem}\label{local to global split}
    Let $(\scr{L},\scr{L}_\leqslant)$
    be an object in $\Be_{C,\mu}$.
    Then for any $(C,\mu)$-generic direction 
    $d\in \R$,
    there exists a 
    unique isomorphism
    \begin{align*}
        \eta_{{\mathbb{I}_d}}\colon (\scr{A}_\mu\otimes\gr\scr{L})_{|{\mathbb{I}_d}}\simeqto \scr{L}_{|{\mathbb{I}_d}}
    \end{align*}
    such that $\gr(\eta_{{\mathbb{I}_d}})=\id$. 
\end{theorem}
\begin{proof} 
    Take a point 
    $c\in  C(\scr{L})\coloneqq \{c'\in\C\mid \gr_{c'}\scr{L}\neq 0\}$. 
    Since we have a short exact sequence 
    \begin{align*}
        0\longrightarrow\scr{L}_{<c}
        \longrightarrow \scr{L}_{\leqslant c}
        \longrightarrow \gr_c\scr{L}\longrightarrow 0,
    \end{align*}
 for any open interval $I\subset S^1$, we have a long exact sequence
    \begin{align*}
        0\longrightarrow H^0(I,\scr{L}_{<c})\longrightarrow H^0(I,\scr{L}_{\leqslant c})
        \longrightarrow H^0(I,\gr_c\scr{L})
        \longrightarrow H^1(I,\scr{L}_{<c})\longrightarrow \cdots.
    \end{align*}
Then the splitting 
$\gr_c\scr{L}\to \scr{L}_{\leqslant c}$
exists if and only if 
the map \[
\delta_{I,c}:
H^0(I,\gr_c\scr{L})\to H^1(I,\scr{L}_{<c})\]
is the zero map.
This map $\delta_{I,c}$ is given as follows:
Take $v\in H^0(I,\gr_c\scr{L})$.
By definition, there exists 
an open covering 
$I=\bigcup_{\alpha=1}^NI_\alpha$ 
and sections $\widetilde{v}_\alpha$ in 
$H^0(I_\alpha,\scr{L}_{\leqslant c})$
such that 
\begin{itemize}
    \item we have $I_\alpha\cap I_\beta=\emptyset $
    if $|\alpha-\beta|> 1$, and 
    \item we have $\widetilde{v}_\alpha
\equiv v_{|I_\alpha}\mod H^0(I_\alpha,\scr{L}_{<c})$
for all $\alpha=1,\dots,N$.
\end{itemize}
Then we have 
\[\delta_{I,c}(v)=\left[\{u_\alpha\coloneqq \widetilde{v}_{\alpha|I_{\alpha}\cap I_{\alpha+1}}-\widetilde{v}_{\alpha+1|I_\alpha\cap I_{\alpha+1}}\}_{\alpha=1}^{N-1}\right].\]  

Let $I_1,I_2$ be open intervals such that 
$\delta_{I_1,c}=\delta_{I_2,c}=0$ for all $c\in \C$.
Assume that $I=I_1\cup I_2$
is contained in ${\mathbb{I}_d}$ for some $d\in\R$.
Then, we shall show that $H^1(I,\scr{L}_{<c})=0$. 
By assumption, we have $H^1(I_j,\scr{L}_{<c})=0$
by Lemma \ref{vanishing A_mu}. 
Again by assumption,
we have $\eta_{I_1|I_0}\colon H^0(I_0,(\scr{A}_\mu\otimes\gr\scr{L})_{<c})\simeqto H^0(I_0,\scr{L}_{<c})$
and the cohomology $H^1(I_0,\scr{L}_{<c})$
is identified with the cokernel of the
map 
\begin{align}\begin{split}\label{I0}
    H^0(I_1,(\scr{A}_\mu\otimes\gr \scr{L})_{<c})\oplus H^0(I_2,(\scr{A}_\mu\otimes\gr \scr{L})_{<c})&\longrightarrow H^0(I_0,
    (\scr{A}_\mu\otimes\gr \scr{L})_{<c}),\\
    (v,w)&\mapsto 
    \varphi(w_{|I_0})-v_{|I_0},    
\end{split}
\end{align}
where we put
$\varphi= 
\eta_{I_1|I_0}^{-1}\circ \eta_{I_2|I_0}$.

We claim that the map \eqref{I0} is surjective, which implies that $H^1(I,\scr{L}_{<c})=0$. 
The proof of this claim will be given 
in Appendix \ref{Appendix}. 
As a consequence, we obtain that for any direction $d$,
we have $H^1({\mathbb{I}_d},\scr{L}_{<c})=0$.

Lastly, we note that for any $(C,\mu)$-generic 
direction $d$, 
we have 
\[H^0({\mathbb{I}_d},\scr{L}_{<c})
\simeq H^0({\mathbb{I}_d},(\scr{A}_\mu\otimes \gr\scr{L})_{<c})=0\]
for any $c\in C(\scr{L})$. 
This implies the uniqueness of 
the splitting. 
\end{proof}

\begin{theorem}
    The category $\Be_{C,\mu}$
    is an abelian category. 
\end{theorem}
\begin{proof}
    The proof is completely parallel to that for standard Stokes-filtered local systems \cite{Sabbah} 
    once we have proved 
    Theorem \ref{local to global split}. 
\end{proof}

\subsection{De Rham functor}\label{functor}
Let $(E,\mathcal{I},\Xi, (E^0,\nabla^0))$ be
an object in $\mathsf{dR}_{C,\mu}$. 
Using notations $\mathcal{I}=(I_k\mid k\in \Lambda_K)$, $\Xi=(\Xi_k)_{k\in\Lambda_K}$,
and $\nabla^k=\Xi_k^{-1}\circ \nabla^0\circ \Xi_k$, we set 
\begin{align*}
    \DR^\mu(E\otimes_{\scr{O}} \scr{A}_{|I_k},\nabla^k)&=\scr{A}_\mu\scr{H}^0{\DR}(E\otimes_{\scr{O}} \scr{A}_{|I_k},\nabla^k).
\end{align*}
Here, on the right hand side, the action
$\mathscr{A}_\mu\subset \widetilde{\mathscr{O}}$
on \[
\scr{H}^0{\DR}(E\otimes_{\scr{O}} \scr{A}_{|I_k},\nabla^k)
\subset E\otimes_{\mathscr{O}}\widetilde{\mathscr{O}}\]
is defined. 
We then define the filtration on $\DR^\mu(E\otimes_{\scr{O}} \scr{A}_{|I_k},\nabla^k)$ as follows:
\begin{align*}
    \DR^\mu_{\leqslant c}(E\otimes_{\scr{O}} \scr{A}_{|I_k},\nabla^k)&=\sum_{c'+c''\leqslant c}
    \scr{A}_\mu ^{\leqslant c'}\scr{H}^0\DR_{\leqslant c''}(E\otimes_{\scr{O}} \scr{A}_{|I_k},\nabla^k).
\end{align*}
\begin{lemma}\label{fildef}
    On $I_{k,k+1}=I_k\cap I_{k+1}$ $(k\in\Lambda_K)$,
    we have the equality \[\DR^\mu_{\leqslant c}(E\otimes_{\scr{O}} \scr{A}_{|I_k},\nabla^k)
    =\DR^\mu_{\leqslant c}(E\otimes_{\scr{O}} \scr{A}_{|I_{k+1}},\nabla^{k+1})
    \]
    as subsheaves in $E\otimes_\scr{O}\widetilde{\mathscr{O}}_{|I_{k,k+1}}$.
\end{lemma}
\begin{proof}
    By the condition \eqref{aut}, the automorphism $\Xi_{k+1}\circ \Xi_k^{-1}$ preserves the filtration. Thus, we have
    \begin{align*}
        \DR^\mu_{\leqslant c}(E\otimes_{\scr{O}} \scr{A}_{|I_k},\nabla^k)
        &=\Xi_k^{-1}[\DR^\mu_{\leqslant c}(E^0,\nabla^0)]\\
        &=(\Xi_k^{-1}\circ \Xi_{k+1}) \circ \Xi_{k+1}^{-1}[\DR^\mu_{\leqslant c}(E^0,\nabla^0)]\\
        &=(\Xi_{k+1}\circ \Xi_k^{-1})^{-1} [\DR^\mu_{\leqslant c}(E\otimes_{\scr{O}} \scr{A}_{|I_{k+1}},\nabla^{k+1})]\\
        &=\DR^\mu_{\leqslant c}(E\otimes_{\scr{O}} \scr{A}_{|I_{k+1}},\nabla^{k+1}),
    \end{align*}
    where the last equality follows from the fact that $\Xi_{k+1}\circ \Xi_k^{-1} \in \aut^{<0}_\mu(E^0)$.
\end{proof}
Then we define a filtered module 
over $\scr{A}_\mu$,
which we denote by $\DR^\mu(E,\Xi)$. 
\begin{lemma}\label{natural}
    For a morphism
    $f=(f,f^0)\in \Hom_{\dR}((E,\Xi),(F,\Theta))$,
    there exists a natural morphism of filtered $\scr{A}_\mu$-modules$\colon$
    \begin{align*}
        f_*\colon \DR^\mu (E,\Xi)\to
        \DR^\mu(F,\Theta).
    \end{align*}
\end{lemma}

\begin{proof}
    We use the notation in Definition \ref{dRE}.
    Set $\nabla^{\Xi,p}\coloneqq \Xi_{\mathsf{u}(p)}^{-1}\nabla^{0}\Xi_{\mathsf{u}(p)}$ and
    $\nabla^{\Theta,p}\coloneqq \Theta_{\mathsf{v}(p)}^{-1}\nabla^{0}\Theta_{\mathsf{v}(p)}$
    for $p\in \Lambda_P$.
    On $I''_p$, we compute
    \begin{align*}
        \nabla^{\Theta,p}\circ f
        &= (\Theta_{\mathsf{v}(p)}^{-1}\nabla^{0} \Theta_{\mathsf{v}(p)}) \circ f \\
        &= \Theta_{\mathsf{v}(p)}^{-1} \nabla^{0} (f^0 \circ \Xi_{\mathsf{u}(p)}) \quad (\text{by \eqref{mor}}) \\
        &= \Theta_{\mathsf{v}(p)}^{-1} f^0 \nabla^{0} \Xi_{\mathsf{u}(p)} \quad (\text{since } f^0 \text{ is flat}) \\
        &= f \circ (\Xi_{\mathsf{u}(p)}^{-1} \nabla^{0} \Xi_{\mathsf{u}(p)}) \\
        &= f \circ \nabla^{\Xi,p}.
    \end{align*}
    Hence, we obtain a morphism
    \[f_{*|I''_p}\colon \DR^\mu(E,\Xi)_{|I''_p}\to \DR^\mu(F,\Theta)_{|I''_p}\]
    for any $p\in\Lambda_P$, which induces the desired global morphism. 
\end{proof}
\begin{proposition}\label{split-exsist}
    We obtain a functor $\DR^\mu\colon \mathsf{dR}_{C,\mu} \to \mathsf{Be}_{C,\mu}$.
\end{proposition}
\begin{proof}
    It remains to show that $\DR^\mu(E,\Xi)$ is an object of $\mathsf{Be}_{C,\mu}$.
    More precisely, we need to verify that it satisfies the conditions in Definition \ref{DefBe}.
    This follows from the isomorphism
    \begin{align*}
        \eta_{I_k}\coloneqq \Xi_k^{-1}\colon
        \DR^\mu(E^0,\nabla^0)_{|I_k}\simeqto
        \DR^\mu (E,\Xi)_{|I_k}
    \end{align*}
    defined for each $k\in\Lambda_K$, together with the fact that $\DR^\mu(E^0,\nabla^0)$ is a graded $\scr{A}_\mu$-module.
\end{proof}
The following  
is the main theorem of this paper. 
\begin{theorem}\label{main theorem}
    The functor 
    \[\DR^\mu\colon \mathsf{dR}_{C,\mu} \to \mathsf{Be}_{C,\mu}\]
    is an equivalence of abelian categories.
\end{theorem}
The proof will be given in 
\S \ref{PROOF}.

\subsection{Existence of analytic lifts and summability}
\begin{proposition}
\label{summability}
   Let $I\subset S^1$ be an open arc. Take objects
    \[(E,\Xi)=(E,\mathcal{I},({\Xi}_k)_{k\in\Lambda_N},(E^0,\nabla^0))\in {\dR}_{C,\mu}\] 
    and $\scr{L}\in \mathsf{Be}_{C,\mu}$ 
    such that  ${\DR}^\mu(E,\Xi)\simeq \scr{L}$. 
    Then, there exists 
    an analytic lift
    \[\Xi_I\colon E\otimes \scr{A}_{|I}\longrightarrow E^0\otimes\scr{A}_{|I}\]
    such that $\Xi_I\circ \Xi_{k}^{-1}\in H^0(I\cap I_k,\aut^{<0}_\mu(E^0))$
    if and only if 
    there exists a filtered isomorphism 
    \[\eta_I\colon (\scr{A}_\mu\otimes \gr(\scr{L}))_{|I}\longrightarrow \scr{L}_{|I}\] 
    such that $\gr(\eta_I)=\id$.     
\end{proposition}
\begin{proof}
    The ``only if'' part is similar to the proof of Proposition \ref{split-exsist}.
    We show the ``if'' part.
    Let $\eta_I$ be as in the statement.
    The composition $\Xi_{k}\circ \eta_{I}$ (restricted to $I\cap I_k$) defines a section of
    \[H^0(I\cap I_k,\DR^\mu_{\leqslant 0}(\End(E^0))).\]
    Since $\gr(\eta_I)=\id$, we find that $\Xi_k\circ\eta_I$ corresponds to a flat section of $\End(E^0)$ modulo rapid decay parts.
    Therefore, $\Xi_k\circ\eta_I$ is asymptotic to the identity, i.e., $\asy(\Xi_k\circ\eta_I)=\id$.
    This implies that $\eta_I^{-1}$ provides the analytic lift of $\widehat{\Xi}_k$ on $I\cap I_k$, as desired.
\end{proof}

\begin{definition}\label{DefSummable}
    An object $(E,\Xi)$ in $\dR_{C,\mu}$ is said to be \textit{summable along $d\in \R$} if the underlying formal isomorphism $\widehat{\Xi}$ is summable along $d$.

    An object $\scr{L} \in \mathsf{Be}_{C,\mu}$ is said to be \textit{summable along $d$} if there exists a positive real number $\varepsilon$ and an isomorphism
    \[\eta_d\colon (\scr{A}_\mu\otimes \gr(\scr{L}))_{|I_d(\varepsilon)}\longrightarrow \scr{L}_{|I_d(\varepsilon)}\]
    satisfying $\gr(\eta_d)=\id$,
    where $I_d(\varepsilon)=\{e^{\i\theta}\in S^1\mid |\theta-d|<(\pi+\varepsilon)/2 \}$.
\end{definition}

\begin{corollary}
    For a $(E,\Xi)\in \dR_{C,\mu}$
    and $d\in\R$,
    $(E,\Xi)$
    is summable along $d\in \R$
    if and only if 
    $\DR^\mu(E,\Xi)$
    is summable along $d$. \qed
\end{corollary}
We remark that if the rank of $L$ is one, any object in $\dR_{C,\mu}$ is summable along any $(C,\mu)$-generic direction $d$ by
the proof of Theorem \ref{local to global split}.

\subsection{Duality and solution functor}
This subsection is only used in \S \ref{Comp1} and is not required for the proof of Theorem \ref{main theorem} (see \S \ref{PROOF}).

Let $j_z^*\colon {\C\conv{z}}\to {\C\conv{z}}$ be an 
automorphism given by $j_z^*(f)(z)=f(-z)$.
For an ${\C\conv{z}}$-vector space $E$,
let $j_z^*E$ be the set $E$ 
with action $f\star v\coloneqq j^*(f)v$
for $f\in {\C\conv{z}}$ and $v\in E$.
 We set $E^\vee=\Hom_{{\C\conv{z}}}(E,{\C\conv{z}})$
and $\mathbb{D}_z(E)\coloneqq j_z^*(E^\vee)=(j^*_zE)^\vee$. A ${\C\conv{z}}$-linear map 
$f\colon E\to F$
induces $j_z^*(f)\colon j_z^*E\to j_z^*F$,
$f^\vee\colon F^\vee\to E^\vee $,
and $\mathbb{D}_z(f)
\colon \mathbb{D}_z(F)\to \mathbb{D}_z(E)$
in a natural way.
If $(E^0,\nabla^0)$ is a meromorphic connection,
then $\mathbb{D}_z(E^0)$ has a natural connection 
$\mathbb{D}_z(\nabla^0)$. 
We set $\mathbb{D}_z(E^0,\nabla^0)=(\mathbb{D}_zE^0,\mathbb{D}_z\nabla^0)$.
\begin{definition}
    We set \[\Sol_z(E^0,\nabla^0)\coloneqq 
    (\scr{H}^0\widetilde{\DR}(\mathbb{D}_z(E^0,\nabla^0)),
    \scr{H}^0\DR_{\leqslant }(\mathbb{D}_z(E^0,\nabla^0))).\] 
\end{definition}

For a formal morphism 
$\widehat{\Xi} \colon E\otimes_\scr{O} \C\jump{z}_1\to F\otimes_{\scr{O}} \C\jump{z}_1$,
we have a morphism 
$\mathbb{D}_z(\widehat{\Xi})\coloneqq j_z^*\widehat{\Xi}^\vee\colon \mathbb{D}_z(F)\otimes_\scr{O}\C\jump{z}_1\to \mathbb{D}_z(E)\otimes\C\jump{z}_1$.

On the boundary circle $S^1$ of $\widetilde{\C}$,
we also have $j_z\colon S^1\to S^1$,
$j_z(e^{\i\theta})=e^{\i (\theta+\pi)}$. 
For a cyclic covering 
$\mathcal{I}=\{I_k\}_{k\in\Lambda_K}$ of $S^1$,
we set $j_z^*\mathcal{I}\coloneqq 
\{j_z(I)\}_{k\in \Lambda_K}$. 
For an $\scr{A}$-module 
$\scr{E}$,
we set $\scr{E}^\vee\coloneqq \homscr(\scr{E},\scr{A})$,
$j^*\scr{E}=\scr{A}\otimes_{j^{-1}\scr{A}}j^{-1} \scr{E}$,
and $\mathbb{D}_z(\scr{E})=j_z^*\scr{E}^\vee$. 
For two $\scr{A}$-modules 
$\scr{E}$ and $\scr{F}$,
an open subset $I\subset S^1$,
and a morphism 
$f_I\colon \scr{E}_{|I}\to
    \scr{F}_{|I}$,
we set $\mathbb{D}_z(f_I)\coloneqq j_z^*f_I^\vee$. 
Let $\mathsf{C}^{\rm op}$
denote the opposite category 
of a category $\mathsf{C}$. 
\begin{definition}
    The duality functor $\mathbb{D}_z\colon \mathsf{dR}_{C,\mu}\to \mathsf{dR}_{C,\mu}^{\rm op}$
    is defined as 
    \begin{align*}
        \mathbb{D}_z((E,\mathcal{I},(\Xi_k)_{k\in \Lambda_K}),(E^0,\nabla^0))\coloneqq 
        ((\mathbb{D}_zE,j_z^*\mathcal{I},(\mathbb{D}_z(\Xi_k)^{-1})_{k\in \Lambda_K}),\mathbb{D}_z(E^0,\nabla^0)),
    \end{align*}
    and $\mathbb{D}_z(f,f^0)\coloneqq 
    (\mathbb{D}_zf,\mathbb{D}_zf^0)$. 
    We also define a solution functor 
    \[\mathrm{Sol}_\mu\colon \mathsf{dR}_{C,\mu}\longrightarrow \mathsf{Be}_{C,\mu}^{\rm op}\]
    by $\mathrm{Sol}_\mu=\DR^\mu\circ
    \mathbb{D}_z$. 
\end{definition} 
\begin{remark}
    It is easy to see that $\mathbb{D}_z\circ\mathbb{D}_z\simeq \id$.
    It is also easy to define a
    functor $\mathbb{D}_z\colon \mathsf{Be}_{C,\mu}\to \mathsf{Be}_{C,\mu}^{\rm op}$
    with a natural isomorphism $\mathrm{Sol}_\mu\simeq \mathbb{D}_z\circ \DR^\mu$. Since these are not
    used in this paper,
    we leave these definitions and proofs to the reader.
\end{remark}

\section{A proof of Theorem $\ref{main theorem}$
}\label{PROOF}
In this section, we prove the Riemann--Hilbert correspondence (Theorem \ref{main theorem}) formulated above. We follow the notation from the previous sections.
\subsection{Internal homomorphisms}
Let $E$
and $F$
be objects in $\mathsf{dR}_{C,\mu}$. 
We use the notation in Definition \ref{dRE}.  
We consider internal homomorphism
\[\hom(E,F)=((\hom_{{\C\conv{z}}}(E,F),\cal{I}'',\Pi),(\hom_{{\C\conv{z}}}(E^0,F^0),\nabla^0)).\]
Here, $\hom_{{\C\conv{z}}}(E,F)$ denotes the ${{\C\conv{z}}}$-vector space of ${{\C\conv{z}}}$-linear maps from $E$ to $F$. 
The connection $\nabla^0$ on $\hom_{{\C\conv{z}}}(E^0,F^0)$
is defined as $\nabla^0(\varphi)=(\nabla^{F^0}\circ \varphi )-(\varphi\circ\nabla^{E^0})$
for $\varphi\in \hom_{{\C\conv{z}}}(E^0,F^0)$. 
$\cal{I''}=(I''_m)_{m\in\Lambda_M}$ denotes a mutual 
refinement of $\mathcal{I}$ and $\cal{I}'$. 
The isomorphism $\Pi=({\Pi}_{m})_{m\in \Lambda_M}$
is defined as  \[{\Pi}_{m}(\psi)={\Theta}_{\mathsf{v}(m)}\circ\psi\circ {\Xi}_{\mathsf{u}(m)}^{-1}\]
for $\psi\in \hom(E,F)\otimes\scr{A}_{|W_m}$
and $m\in \Lambda_M$. 

We have $\hom(E,F)\in \dR_{\Delta(C),\mu}$
for $\Delta(C)=\{c-c'\mid c,c'\in C\}$. 
We note that
$\Delta(C)$ possibly have 
elements $d,d'\in \Delta(C)$ and $\ell\in L$
with $d-d'=\mu(\ell)$. 
However, this does not affect the construction of $\DR^\mu$. Thus, $\DR^\mu(\hom(E,F))$ is well-defined.
\begin{lemma}
    The isomorphism class of $\hom(E,F)$
    does not depend on the choice
    of mutual refinement $\cal{I}''$. \qed
\end{lemma} 

\begin{theorem}\label{homiso}
    We have a natural isomorphism
    \begin{align*}    \Hom_{\mathsf{dR}} (E,F)\simeqto 
    H^0(S^1,\DR^\mu_{\leqslant0}(\hom(E,F))).
    \end{align*}
\end{theorem}
\begin{proof}
    Let $(f,f^{0})$ be an element in 
    $\Hom_{\mathsf{dR}} (E,F)$. 
    By definition,
    there is a mutual  refinement 
    $(\cal{I}'',\mathsf{u},\mathsf{v})$
    of $\mathcal{I}$ and $\cal{I}'$
    such that 
    we have 
    $\Theta_{\mathsf{v}(p)}\circ f=f^0\circ \Xi_{\mathsf{u}(p)},$
    which implies
    \[\nabla^{\Pi,p}(f)=0,\] where
    $\nabla^{\Pi,p}$ denotes the induced connection on $\hom(E,F)_{|I''_p}$.
    Indeed, we can compute
    \begin{align*}
        \nabla^{\Pi,p}(f)&=
        \Pi_p^{-1}\circ \nabla^{0}\circ \Pi_{p}(f)\\
        &=\Pi_p^{-1}(\nabla^0\circ (\Theta_{\mathsf{v}(p)}\circ f\circ \Xi_{\mathsf{u}(p)}^{-1})-
        (\Theta_{\mathsf{v}(p)}\circ f\circ \Xi_{\mathsf{u}(p)}^{-1})\circ \nabla^0)\\
        &=(\Theta_{\mathsf{v}(p)}^{-1}\nabla^0 \Theta_{\mathsf{v}(p)})\circ f - f\circ (\Xi_{\mathsf{u}(p)}^{-1}\nabla^0 \Xi_{\mathsf{u}(p)})\\
        &=\nabla^{\Theta,p}\circ f-f\circ \nabla^{\Xi,p}.
    \end{align*}
    The last term vanishes by the same argument as in the proof of Lemma \ref{natural}.
    
    It implies that \[f_{|I''_p}\in H^0(I''_p,\scr{H}^0{\DR}_{\leqslant 0}(\hom(E,F),\nabla^p)).\]
    Hence it defines a section of 
    $H^0(S^1,\DR^\mu(\hom(E,F))_{\leqslant0})$.

    We shall construct the inverse of 
    this map. 
    Let $g$ be a section of the vector space
    $H^0(S^1,\DR^\mu(\hom(E,F))_{\leqslant0})$.
    Since $g^0_p\coloneqq \Pi_{p}(g_{|I''_p})$
    satisfies $\gr_0(g^0_p)=\gr_0(g^0_{p+1})$
    on $I''_{p,p+1}=I''_p\cap I''_{p+1}$ for any $p\in\Lambda_P$, 
    we obtain a global section 
    \[f^0_g\in H^0(S^1,\gr_0\DR^\mu(\hom(E^0,F^0)).\]
    Since we have 
    \begin{align*}
        H^0(S^1,\gr_0\DR^\mu(\hom(E^0,F^0))
        &=H^0(S^1,\scr{H}^0\DR_{\gr_0}(\hom(E^0,F^0))
        \\
        &\simeq H^0(S^1,\scr{H}^0\DR_{\leqslant 0}(\hom(E^0,F^0)))\\
        &\simeq H^0\DR(\hom(E^0,F^0))\\
        &\simeq \Hom (E^0,F^0)^{\nabla^0},
    \end{align*}
    we obtain a flat morphism$f_g^0\in \Hom(E^0,F^0)^{\nabla^0}$. 
    Here, the first equality follows from the following assumptions:
    \begin{itemize}
        \item The exponential factors of $E^0$ and $F^0$ is contained in $C$.
        \item We have 
        $c+\mu(L)\neq c'+\mu(L)$
    for any $c,c'\in C$ with $c\neq c'$. 
    \end{itemize}
    The last three isomorphisms 
    follow from the general theory
    of meromorphic connections \cite{Sabbah}. 
    We then note that 
    $g\in H^0(S^1,\hom(E,F)\otimes \scr{A}^{\leqslant 0})=\Hom_{{\C\conv{z}}}(E,F)$.

    It remains to show that 
    $\Theta_{\mathsf{v}(p)}\circ g=f_g^0\circ \Xi_{\mathsf{u}(p)}$
    for any $p\in\Lambda_P$.
    By Theorem \ref{local to global split} and Lemma \ref{summability},
    there is a lift $\Xi^d$ and $\Theta^d$ on 
    ${\mathbb{I}_d}$ for every $(C,\mu)$-generic direction
    $d\in\R$. 
    Then, it is easy to see that 
    $\Pi^d(g_{|{\mathbb{I}_d}})=f_{g|{\mathbb{I}_d}}^0$,
    which implies the claim. 
\end{proof}
Similarly, for two objects $\scr{L},\scr{M}$ in $\mathsf{Be}_{C,\mu}$,
the sheaf of homomorphisms \[\homscr(\scr{L},\scr{M})\]
is naturally equipped with a filtration and admits a local splitting.

\begin{lemma}\label{Bettihom}
    For two objects $\scr{L},\scr{M}$ in 
$\mathsf{Be}_{C,\mu}$,
we have 
\begin{align*}
    \mathrm{Hom}_{\mathsf{Be}}(\scr{L},\scr{M})\simeqto H^0(S^1,\homscr(\scr{L},\scr{M})_{\leqslant 0}).
\end{align*}
\end{lemma}
\begin{proof}
    Straightforward. 
\end{proof}
\subsection{Full and faithful}\label{FF}
Let $E=(E,{\Xi}),F=(F,{\Theta})$ be objects in $\mathsf{dR}_{C,\mu}$. 
We have a natural morphism between internal homs:
\[\DR^\mu(\hom(E,F))\to \homscr(\DR^\mu(E),\DR^\mu(F)).\]
Locally on ${\mathbb{I}_d}$, this is identified with the map
\[\scr{A}_\mu \widetilde{\DR}(\hom(E^0,F^0))\to \scr{A}_\mu \homscr(\widetilde{\DR}(E^0),\widetilde{\DR}(F^0))\]
via $\Pi^p$. This local map is known to be an isomorphism \cite{Sabbah}. 
Since this isomorphism preserves the filtration, taking global sections with  $\leqslant 0$ yields
\begin{align*}
    H^0(S^1,\DR^\mu_{\leqslant 0}(\hom(E,F)))
    \simeqto
    H^0(S^1,\homscr(\DR^\mu(E),\DR^\mu(F))_{\leqslant 0}).
\end{align*}

By Theorem \ref{homiso} and Lemma \ref{Bettihom}, this induces an isomorphism
\begin{align*}
    \Hom_{\mathsf{dR}_{C,\mu}}(E,F)\xrightarrow{\sim} \Hom_{\mathsf{Be}_{C,\mu}}(\DR^\mu(E),\DR^\mu(F)).
\end{align*}
This proves that the functor $\DR^\mu$ is fully faithful.

\subsection{Essential surjectivity}\label{EssS}
We set $\scr{E}nd^{<0}(\scr{L})\coloneqq \homscr(\scr{L},\scr{L})_{<0}$
for $\scr{L}\in\Be_{C,\mu}$.
\begin{theorem}\label{Classify}
    Let $\scr{G}$ be a graded
    $\scr{A}_\mu$-module in $\mathsf{Be}_{C,\mu}$. 
    Then there is a canonical one-to-one correspondence between:
    \begin{itemize}
        \item the set of isomorphism classes of  
        pairs $((\scr{L},\scr{L}_{\leqslant 0}),\xi)$ of an object $(\scr{L},\scr{L}_{\leqslant 0})$ in 
        ${\mathrm{Be}}_{C,\mu}$
     and an isomorphism 
     \[\xi\colon \scr{A}_\mu\otimes\gr\scr{L}
     \simeqto \scr{G}\]
     of filtered $\scr{A}_\mu$-modules, and 
        \item the cohomology set 
    \begin{align*}
        H^1({S}^1,\aut^{<0}_{\mu}(\scr{G})),
    \end{align*}
    where we put 
    $\aut^{<0}_{\mu}(\scr{G})
    =\id_\scr{G}
    +\scr{E}nd^{<0}(\scr{G})$. 
    \end{itemize}
\end{theorem}
\begin{proof}
     The proof is completely analogous to that for standard Stokes-filtered local systems.
\end{proof}
\begin{theorem}[Essential surjectivity]
    For any object $\scr{L}\in\mathsf{Be}_{C,\mu}$,
    there exists an object
    $(E,\Xi)=((E,\mathcal{I},\Xi),(E^0,\nabla^0))\in {\dR}_{C,\mu}$
    such that ${\DR}^\mu(E,\Xi)\simeq \scr{L}$.
\end{theorem}
\begin{proof}
    For $\scr{L}\in \mathsf{Be}_{C,\mu}$,
    consider the graded Stokes-filtered local system \[\bigoplus_{c\in-C}\gr_c(\scr{L}).\]
    There exists an elementary exponential
    meromorphic connection $(E^0,\nabla^0)$
    such that
    \[(\scr{H}^0\widetilde{\DR}(E^0,\nabla^0),\scr{H}^0\DR_{\leqslant}(E^0,\nabla^0))\simeq \bigoplus_{c\in -C}\gr_c\scr{L}\]
    as Stokes-filtered local systems.
    Consequently, we have
    $\DR^\mu(E^0,\nabla^0)\simeq \scr{G}$ and
    \[\aut_\mu^{<0}(E^0,\nabla^0)
    \simeq \aut_\mu^{<0}(\scr{G}),\]
    where we set $\scr{G}\coloneqq \scr{A}_\mu\otimes_{\C} \gr\scr{L}$.
    This induces an isomorphism
    \[H^1(S^1,\aut^{<0}_\mu(E^0,\nabla^0))\simeq H^1(S^1,\aut^{<0}_\mu(\scr{G})).\]
    Let $[\scr{L}]\in H^1(S^1,\aut^{<0}_\mu(\scr{G}))$
    be the class of $\scr{L}$ given by the
    correspondence in Theorem \ref{Classify},
    and let $[(E,\Xi)]\in H^1(S^1, \aut_\mu^{<0}(E^0))$ be the corresponding class via the above isomorphism.
    We have natural inclusions:
    \begin{align*}
        \aut_\mu^{<0}(E^0)&\hookrightarrow
        \scr{G}^{0}_1(E^0)
        \coloneqq
        \id_{E^0}+\End(E^0)\otimes\scr{A}^{<0}, \\
        \scr{G}^{0}_1(E^0)&
        \hookrightarrow\scr{G}_1(E^0)\coloneqq
        \id_{E^0}+\End(E^0)\otimes\scr{A}.
    \end{align*}
    By the Malgrange--Sibuya theorem for Gevrey order one matrices
    \cite{loday2016divergent}*{\S 3.5.4, Theorem 1.4.2}, the map
    \begin{align*}
        H^1(S^1,\scr{G}^0_1(E^0))\longrightarrow
        H^1(S^1,\scr{G}_1(E^0))
    \end{align*}
    is trivial.
    Consequently, the class $[(E,\Xi)]$ is represented by an object
    \[(E,\Xi)=((E,\mathcal{I},\Xi),(E^0,\nabla^0))\in {\mathsf{dR}}_{C,\mu}.\]
    By construction,
    we have $\DR^\mu(E,\Xi)\simeq \scr{L}$,
    which completes the proof.
\end{proof}

\section{Comparison of isomorphisms  
in dimension one}\label{Comp1}
In this section,
we construct objects in $\mathsf{dR}_{C,\mu}$ and $\mathsf{Be}_{C,\mu}$
arising from complex geometry of dimension one.
We then show that
these objects correspond to each other via the functor $\mathrm{Sol}_\mu$.
We suggest that 
this fact could be seen as 
a reformulation of
the comparison of isomorphisms conjecture
\cite{kontsevich2024holomorphic}*{Conjecture 4.7.1} in this setting. 

\subsection{Setting}\label{setting}
Let $X$ be a compact Riemann surface.
Let $\alpha$ be a meromorphic one-form on $X$.
Let $D$ and $Z$ denote the sets of points in $X$
where $\alpha$ has poles and zeros, respectively. 
Set $Y\coloneqq X\setminus D$.
Let $\widetilde{L}=H_1(Y,\Z)$
be the first homology group
of $Y$.
Let $\widetilde{\mu}_\alpha\colon \widetilde{L}\to \C$
denote the period map, defined by
\begin{align*}
    \widetilde{\mu}_\alpha(\gamma)=\int_\gamma \alpha
    \quad \quad(\gamma\in \widetilde{L}).
\end{align*}
We set $L\coloneqq \widetilde{L}/\mathrm{Ker}(\widetilde{\mu}_\alpha)$.
Let $\mu_\alpha\colon L\to \mathbb{C}$
denote the induced map.
We fix a norm $\|\cdot\|$ on $L$.
In this paper,
we make the following assumption:
\begin{assumption}\label{Assume}
    The sets $D$ and $Z$ are non-empty finite subsets of $X$, and
    the period map $\mu=\mu_\alpha$
    satisfies the support property \eqref{support property}.
\end{assumption}

Let $\pi_Y\colon \widetilde{Y}\to Y$
be the universal covering map of $Y$.
We fix a holomorphic function 
$f\colon \widetilde{Y}\to\C$
such that $df=\pi_Y^*\alpha$.
There exists a finite subset $C_f\subset \C$
such that the set of critical values of $f$ coincides with $C_f+\mu_\alpha(L)$,
and that for any $c,c'\in C_f$,
$c\equiv c' \pmod{\mu_\alpha(L)}$ implies $c=c'$.
Note that $f(\pi_Y^{-1}(Z))=C_f+\mu_\alpha(L)$.
 
For later use,
we fix labelings
$D=\{p_1,\dots,p_r\}$ and
$Z=\{q_1,\dots,q_s\}$.
Let $n_k$ be the order of the pole of $\alpha$ at $p_k$.
We fix a local coordinate chart 
$(V_k, v_k)$ centered at $p_k$ such that $\alpha_{|V_k}=v_k^{-n_k}dv_k$ if $n_k>1$ and 
$\alpha_{|V_k}=\alpha_kv_k^{-1}dv_k$ 
if $n_k=1$.
We have
\begin{align*}
    2\pi\i \alpha_k\in \mu_\alpha(L)
\end{align*}
and $\alpha_k\neq 0$ if $n_k=1$.
Let $m_j$ be the order of the zero of $\alpha$
at $q_j$.
We also fix a local coordinate chart
$(U_j,u_j)$ centered at $q_j$
such that $\alpha_{|U_j}=u_j^{m_j}du_j$. 
Note that we have the equality
\begin{align}\label{genus}
   \sum_{j=1}^sm_j-\sum_{k=1}^r n_k=2g-2 
\end{align}
where $g$ denotes the genus of $X$,
since any meromorphic $1$-form defines a canonical divisor, whose degree is $2g-2$.
\subsection{De Rham object} 
The objects and results
in this subsection 
have already appeared in the work of 
Kontsevich--Soibelman.

Let $\scr{O}_X(*D)$ denote 
the sheaf of meromorphic functions on $X$
whose poles are contained in $D$. 
Let $\Omega^1_{X}(*D)$ denote the 
sheaf of meromorphic one-forms
on $X$ whose poles are contained in $D$.
\begin{definition}[Global de Rham cohomology]
    We define a $\C\conv{z}$-vector space
    $E_{X,\alpha}$ called the  
    \textit{global de Rham cohomology of $(X,\alpha)$} as follows:
    \begin{align*}
    E_{X,\alpha}\coloneqq \Cok[H^0(X,\scr{O}_X(*D))\otimes_\C {\C\conv{z}}\xrightarrow{zd+\alpha}
    H^0(X,\Omega^1_X(*D))\otimes_\C {\C\conv{z}}
    ].
\end{align*}
For $\omega\in H^0(X,\Omega_X^1(*D))$,
we let $[\omega]$ denote the corresponding 
element in $E_{X,\alpha}$. 
\end{definition}

\begin{lemma}\label{fin dim}The vector space
    $E_{X,\alpha}$ is finite dimensional over ${\C\conv{z}}$. 
\end{lemma}
\begin{proof}
Set $\cal{X}=\C\times X$
and $\cal{D}=\C\times D$.  
Let $p_{\cal{X}}\colon \cal{X}\to X$
be the projection. 
We set \[\scr{O}_{\cal{X}}(*\cal{D})=p_{\cal{X}}^{*}\scr{O}_X(*D)
=\scr{O}_{\cal{X}}\otimes_{p_{\cal{X}}^{-1}\scr{O}_X}\scr{O}_X(*D)\] and 
$\Omega_{\cal{X}/\C}^1(*\cal{D})\coloneqq p_{\cal{X}}^*\Omega_{X}^1(*D)$.
Set $\cal{X}^0=\{0\}\times X$.
Let $i_{\cal{X}^0}\colon \cal{X}^0\to\cal{X}$
be the inclusion. 
Set $\scr{O}_{\mathcal{X}^0}\coloneqq i_{\mathcal{X}^0}^{-1}\scr{O}_{\mathcal{X}}$.
Let $\pi_{\cal{X}^0}\colon \cal{X}^0\to \{0\}$
be the projection.

We consider the following complex 
concentrated in degrees 0 and 1:
\begin{align*}
    \cal{C}_{X,\alpha}:
    i^{-1}_{\cal{X}^0}\scr{O}_{\cal{X}}(*\cal{D}) \xrightarrow{zd+\alpha}i_{\cal{X}^0}^{-1}\Omega^1_{\cal{X}/\C}(*\cal{D}).
\end{align*} 
For each integer $\ell$,
we set $\scr{O}_{\alpha,\ell}= \scr{O}(\ell D)$
and $\Omega^1_{\alpha,\ell}=\Omega^1_X(\sum_k(\ell+n_k)p_k)$. 
We then put $\scr{O}_{\cal{X}^0,\alpha,\ell}
\coloneqq i_{\cal{X}^0}^{-1}p_{\cal{X}}^*\scr{O}_{\alpha,\ell}$ and $\Omega^1_{\cal{X}^0,\alpha,\ell}
\coloneqq i_{\cal{X}^0}^{-1}p_{\cal{X}}^*\Omega^1_{\alpha,\ell}$.
We have a sub-complex 
\begin{align*}
    \cal{C}_{\alpha,\ell}\colon \scr{O}_{\cal{X}^0,\alpha,\ell}\xrightarrow{zd+\alpha}\Omega^1_{\cal{X}^0,\alpha,\ell}.
\end{align*}
Then the inclusion $\mathcal{C}_{\alpha,\ell}\hookrightarrow \mathcal{C}_{X,\alpha}$ is a quasi-isomorphism. 
Indeed, at each point $p_k\in D$,
using the local coordinate $(V_k,v_k)$, 
the differential $zd+\alpha$
is described as follows: 
Consider an element
$h(z,v_k)=v_k^mg(z,v_k)\in i_{\cal{X}^0}^{-1}\scr{O}_{\cal{X}}(*\cal{D})_{(0,p_k)}$ with
$g(z,v_k)\in i_{\cal{X}^0}^{-1}\scr{O}_{\cal{X}}$,
$g(z,0)\in \scr{O}\setminus \{0\}$, and $m\in\Z$. We have
\begin{align}
   \begin{split}\label{explicit}
        h(z,v_k)&\mapsto \{z\partial_{v_k}h(z,v_k)+v_k^{-n_k}h(z,v_k)\}dv_k\\
    &=\{zmv_k^{m-1}g(z,v_k)+v_k^mz\partial_{v_{k}}g(z,v_k)+
    v_k^{m-n_k}g(z,v_k)\}dv_k
    \end{split}
\end{align}
if $n_k>1$, and 
\begin{align}
   \begin{split}\label{explicit1}
        h(z,v_k)&\mapsto \{z\partial_{v_k}h(z,v_k)+\alpha_kv_k^{-1}h(z,v_k)\}dv_k\\
    &=\{ (zm+\alpha_k)
    v_k^{m-1}g(z,v_k)+
    v_k^mz\partial_{v_{k}}g(z,v_k)\}dv_k
    \end{split}
\end{align}
if $n_k=1$. 
It follows that $\scr{H}^0(\cal{C}_{X,\alpha})
=\scr{H}^0(\cal{C}_{\alpha,\ell})=0$.
Since we may reduce the pole order 
along $\{v_k=0\}$
by the image of the maps \eqref{explicit}
and \eqref{explicit1}, 
we also have
\[\scr{H}^1(\cal{C}_{\alpha,\ell})\simeqto \scr{H}^1(\cal{C}_{X,\alpha}).\] 
Here, we note that in the case $n_k=1$, the term $(zm+\alpha_k)$ is an invertible element in $\scr{O}$ since $\alpha_k\neq 0$.

For sufficiently large $\ell\in \Z$,
we have that 
$H^1(X,\cal{O_{\alpha,\ell}})=H^1(X,\Omega^1_{\alpha,\ell} )=0$.
Hence we obtain
\begin{align*}
E_{X,\alpha}
\simeq 
\Cok\left[H^0(X,\scr{O}_{\alpha,\ell})
\otimes {\C\conv{z}}\xrightarrow{zd+\alpha} 
H^0(X,\Omega^1_{\alpha,\ell})
\otimes {\C\conv{z}} \right].
\end{align*}
Since $H^0(X,\Omega^1_{\alpha,\ell})$ is finite dimensional,
the lemma follows. 
\end{proof}
\begin{remark}[Relation to \cite{kontsevich2024holomorphic}
with notations in \S \ref{Intro CIC}]
    By the proof of Lemma \ref{fin dim}, we have 
    \[E_{X,\alpha}\simeq H^1_{\dR,\gl}(X,\alpha)\otimes \C\conv{z}\]
    and \[H^0_{\dR,\gl}(X,\alpha)\otimes \C\conv{z}=0.\] 
    Here, we used the notaions 
    in \S \ref{Intro CIC}. 
\end{remark}

\begin{definition}
    We set
    \begin{align*}
        \widehat{E}_{X,\alpha}^0\coloneqq
        \mathrm{Coker}\left[
        \bigoplus_{j=1}^s\C\jump{z,u_j}[z^{-1}]
        \xrightarrow{\bigoplus_j zd+u_j^{m_j}du_j}
        \bigoplus_{j=1}^s\C\jump{z,u_j}[z^{-1}]
        du_j
        \right].
    \end{align*}
    Here, $\{(U_j,u_j)\}_j$ denotes the
    local coordinate systems fixed in this section.
    For $k=0,1,2,\dots$, we let
    $[u^k_jdu_j]$ denote the element in $\widehat{E}^0_{X,\alpha}$ represented
    by $u^k_jdu_j$.
\end{definition}
Let ${\C\pole{z}}$
denote the field of formal Laurent series.
\begin{lemma}
    We have
    \[\widehat{E}_{X,\alpha}^0=
    \bigoplus_{j=1}^s\bigoplus_{k=0}^{m_j-1}\C\pole{z}
    [u_j^kdu_j].\]
    In particular, $\dim_{{\C\pole{z}}}\widehat{E}^0_{X,\alpha} =\sum_{j=1}^sm_j$.
\end{lemma}
\begin{proof}
    In the complex
    \begin{align*}
       \C\jump{z,u_j} \xrightarrow{zd+u_j^{m_j}du_j}\C\jump{z,u_j}du_j,
    \end{align*}
    the map is given by
    \begin{align}\label{local zero}
        u_j^n\mapsto znu_j^{n-1}du_j+u_j^{n+m_j}du_j
    \end{align}
    for $n=0,1,2,\dots$. Hence the cokernel
    has a basis
    \[\{[u_j^kdu_j]\mid k=0,1,\dots,m_j-1\}.\]
    The lemma follows.
\end{proof}

\begin{definition}
    We set $E_{X,\alpha}^0\coloneqq \bigoplus_jE^{0,j}_{X,\alpha}$ with $E^{0,j}_{X,\alpha}\coloneqq\bigoplus_{k=0}^{m_j-1}{{\C\conv{z}}}[u_j^kdu_j]$.
\end{definition}

    For each point $q_j\in Z$,
setting $f_j\coloneqq (m_j+1)^{-1}u_j^{m_j+1}$, we have $df_j=u_j^{m_j}du_j$.
We consider the following commutative diagram:
    \begin{align*}
    \xymatrix@C=5em{
       z\C\jump{z,u_j}
       \ar[d]_{z^2\partial_z-f_j}
    \ar[r]^{d+z^{-1}u_j^{m_j}du_j}
    &\C\jump{z,u_j}du_j
    \ar[d]^{z^2\partial_z-f_j}\\
    z\C\jump{z,u_j}
    \ar[r]^{d+z^{-1}u_j^{m_j}du_j}
       &\C\jump{z,u_j}du_j}
    \end{align*}
    It induces a connection
    $\nabla^j\colon\widehat{E}_{X,\alpha}^{0,j}\to \widehat{E}^{0,j}_{X,\alpha}dz$.
    More explicitly, we have 
    \begin{align*}
        \nabla^j_{z^2\partial_z}[u^k_jdu_j]&=-[f_j(u_j)u_j^kdu_j]\\
        &=
            (m_j+1)^{-1}z(k+1)[u_j^{k}du_j]
    \end{align*} 
    Hence $\nabla^j$ is a regular singular connection
    on $E_{X,\alpha}^{0,j}$ 
    with a logarithmic lattice 
    \[\bigoplus_{k=0}^{m_j-1}
    \scr{O}[u_j^kdu_j].\]
    This property characterizes $E_{X,\alpha}^0\subset 
    \widehat{E}_{X,\alpha}^0$. 
    
   Choose points $\widetilde{q}_j\in \widetilde{Y}$ ($j=1,\dots,s$)
such that $\pi_Y(\widetilde{q}_j)=q_j$.
Recall that we have fixed 
a finite set $C_f\subset \C$ in \S \ref{setting}.
We may assume that the set of critical values corresponds to $\{f(\widetilde{q}_j)\mid j=1,\dots,s\}=C_f$.
Let $c_j \coloneqq f(\widetilde{q}_j)$.

\begin{definition}[Local de Rham cohomology]
    We set
    \[\nabla^0=\bigoplus_{j=1}^s \left(\nabla^j-c_j\cdot \mathrm{id}_{E^{0,j}_{X,\alpha}}\frac{dz}{z^2}\right)\colon E_{X,\alpha}^0\longrightarrow E_{X,\alpha}^0dz.\]
    We call the elementary exponential connection $(E_{X,\alpha}^0,\nabla^0)$
    with exponential factors $C_f$
    the \textit{local de Rham cohomology} of $(X,\alpha)$.
\end{definition}

\begin{lemma}[\cite{kontsevich2024holomorphic}*{Proposition 3.3.2}]\label{formaldR}
    There is a natural isomorphism 
    \[\widehat{\Xi}_{X,\alpha}\colon 
    E_{X,\alpha}\otimes_{\scr{O}}\C\jump{z}\longrightarrow  \widehat{E}^0_{X,\alpha}.\] 
\end{lemma}
\begin{proof}
We use the notation in the proof of 
Lemma \ref{fin dim}.
Let 
$\scr{I}_{\cal{X}^0}=z\scr{O}_{\cal{X}^0}$
be the defining ideal of $\cal{X}^0$.
Then we consider the complex
\begin{align*}
    \mathcal{C}_{\alpha,\ell}^N \colon
    \scr{O}_{\mathcal{X}^0,\alpha,\ell}/\scr{I}_{\mathcal{X}^0}^N
    \xrightarrow{zd+\alpha}
    \Omega_{\mathcal{X}^0,\alpha,\ell}^1/\scr{I}_{\mathcal{X}^0}^N.
\end{align*}
This is isomorphic to the complex
$\scr{O}_{\alpha,\ell}\otimes \C[z]/(z^N)\to \Omega_{\alpha,\ell}^1\otimes \C[z]/(z^N)$.
The limit $\widehat{\mathcal{C}}_{\alpha,\ell} \coloneqq \varprojlim_N \mathcal{C}_{\alpha,\ell}^N$
is isomorphic to the complex
\begin{align*}
    \scr{O}_{\alpha,\ell}\jump{z}\longrightarrow \Omega_{\alpha,\ell}^1\jump{z}.
\end{align*}
There is a natural morphism
$\mathcal{C}_{\alpha,\ell}\to\widehat{\mathcal{C}}_{\alpha,\ell}$
which induces a morphism
\[\widehat{\Xi}_1 \colon (\R^1\pi_{\mathcal{X}^0*}\mathcal{C}_{\alpha,\ell})
\otimes \C\jump{z}\longrightarrow \R^1\pi_{\mathcal{X}^0*}\widehat{\mathcal{C}}_{\alpha,\ell}.
\]
One can check that $\widehat{\Xi}_1$ is an isomorphism.

Let $\C\jump{z,u_j}_{q_j}$
denote the skyscraper sheaf supported on $q_j$
whose fiber is $\C\jump{z,u_j}$. 
The Taylor expansion at each point $q_j$ defines a morphism of complexes
\begin{align}\begin{split}\label{Taylor}
\xymatrix@C=5em{
\widehat{\cal{C}}_{\alpha,\ell}
\colon\scr{O}_{\alpha,\ell}\jump{z}\ar[r]\ar[d]_{\tau^0}& \Omega_{\alpha,\ell}\jump{z}\ar[d]^{\tau^1}
\\ 
\cal{T}\colon \bigoplus_{j=1}^s\C\jump{z,u_j}_{q_j}\ar[r]^{zd+u_j^{m_j}du_j}&
\bigoplus_{j=1}^s\C\jump{z,u_j}_{q_j}du_j
}.
\end{split}
\end{align}
Here, $\tau^1(\omega)=\sum_j g_j(u_j)du_j$
if we have $\omega_{|U_j}=g_j(u_j)du_j$
for $g_j(u_j)\in \C\{u_j\}$. 
Note that $\scr{H}^0(\widehat{\mathcal{C}}_{\alpha,\ell})=0$,
and $\scr{H}^1(\widehat{\mathcal{C}}_{\alpha,\ell})$ is supported on $Z$.
(The latter fact follows from $\scr{H}^1(\mathcal{C}_{\alpha,\ell}^N)_{P}=0$
for any $N$ and $P\notin Z$; see \eqref{explicit}).
Then the local description \eqref{local zero}
implies that \eqref{Taylor} is a quasi-isomorphism. We obtain an isomorphism
\begin{align*}
    \widehat{\Xi}_2\colon \R^1\pi_{\cal{X}^0*}\widehat{\cal{C}}_{\alpha,\ell}\longrightarrow\R^1\pi_{\cal{X}^0*}\cal{T}.
\end{align*}
The composition $\widehat{\Xi}_2\circ\widehat{\Xi}_1$
induces the desired isomorphism. 
\end{proof}
\begin{remark}[Relation to \cite{kontsevich2024holomorphic}
with notations in \S \ref{Intro CIC}]
    By the proof of Lemma \ref{formaldR},
    we have 
    \[\widehat{E}_{X,\alpha}^0\simeq H^1_{\dR,\loc}(X,\alpha)\otimes_{\C\jump{z}}\C\pole{z}\]
    and 
    $H^0_{\dR,\loc}({X,\alpha})=0$
    in this case. 
\end{remark}
The following theorem gives a
more precise description of 
$\phi_\dR$ in \S \ref{Intro CIC},

\begin{theorem}[De Rham local to global isomorphism]\label{Xihat}
    The isomorphism $\widehat{\Xi}_{X,\alpha}$
    is defined over $\C\jump{z}_{1}$,
    i.e., we have
    \begin{align*}
        \widehat{\Xi}_{X,\alpha}\colon E_{X,\alpha}\otimes\C\jump{z}_1\xrightarrow{\sim}
        E^0_{X,\alpha}\otimes\C\jump{z}_1.
    \end{align*}
\end{theorem}
\begin{proof}
Take a form $\omega\in H^0(X,\Omega^1_X(*D))$.
Locally, we have the expression
\[\omega_{|U_j}=g_j(u_j)du_j\]
for some $g_j(u_j)\in\C\{u_j\}$.
Then, by \eqref{local zero} and the construction of $\widehat{\Xi}_{X,\alpha}$
in the proof of Lemma \ref{formaldR},
we have
\[\widehat{\Xi}_{X,\alpha}([\omega])=\sum_{j=1}^s
\sum_{k=0}^{m_j-1}\widehat{g}_{j,k}(z)[u_j^kdu_j],\]
where $\widehat{g}_{j,k}(z)$ is given as follows:
If $g_j(u_j)=\sum_n a_n^{(j)}u_j^n$ ($a_n^{(j)}\in \C$),
then
\begin{align*}
  \widehat{g}_{j,k}(z)&=
\sum_{n=0}^\infty(-1)^na_{m_jn+k}^{(j)}\prod_{\ell=0}^{n-1} (k+m_j\ell+1)
z^{n}\\
&=\Gamma\left(\frac{k+1}{m_j}\right)^{-1}\sum_{n=0}^\infty(-1)^na^{(j)}_{m_jn+k}m_j^n\Gamma\left(\frac{k+1}{m_j}+n\right)z^n.
\end{align*}
We see that $\widehat{g}_{j,k}(z)\in \C\jump{z}_1$
for any $j,k$.
The theorem follows.
\end{proof}

\subsection{Local objects}
We shall recall the comparison theorem
for local objects associated with $(X,\alpha)$.
The de Rham local object $(E^0_{X,\alpha},\nabla^0)$
has already been defined above.
We now define the corresponding local Betti structure.

Let $\{(U_j,u_j)\}_{j=1}^s$
be the local coordinate systems fixed in \S \ref{setting}.
Via the coordinate function $u_j$,
we regard $U_j$ as an open subset of the
complex plane, which we denote by $\C_{u_j}$.
We consider the real blow-up
$\widetilde{\P}^1_{u_j}$
of $\P^1_{u_j}=\C_{u_j}\cup\{\infty_j\}$
at $\infty_j$.
The boundary $S^1_j=\{e^{\mathrm{i}\theta_j}\mid \theta_j\in\R\}$
parametrizes the directions at $\infty_j$.
For $e^{\mathrm{i}\theta}\in S^1$,
we set
\begin{align*}
    \widetilde{D}_j^\theta=\{e^{\i \theta_j}\in S^1_j\mid
    \mathrm{Re}[\exp(-2\pi\i\{(m_j+1)\theta_j+\theta\})]>0\}.
\end{align*}
Then, there is a local system
of $\C$-vector spaces
$\scr{F}_{X,\alpha}^{j}$
on $S^1$ such that
\[(\scr{F}_{X,\alpha}^{j})_{e^{\mathrm{i}\theta}}=H_1(\widetilde{\P}^1_{u_j},\widetilde{D}^\theta_j;\C),\]
where the right-hand side denotes the relative homology group.
More explicitly, the local system $\scr{F}^j_{X,\alpha}$ is given as follows:
For an open arc $I\subset S^1$,
we set $\widetilde{D}^I_j=\bigcap_{e^{\mathrm{i}\theta}\in I}\widetilde{D}^\theta_j.$
Then, we obtain a presheaf
$I\mapsto H_1(\widetilde{\P}^1_{u_j},\widetilde{D}_j^I;\C).$
The sheaf associated with
this presheaf is defined to be $\scr{F}^j_{X,\alpha}$.

We then set \[\scr{F}_{X,\alpha}=\bigoplus_{j=1}^s\scr{F}^{j}_{X,\alpha}.\]
Recall that we have fixed complex numbers
$C_f=\{c_j\mid j=1,\dots,s\}$.
We regard $\scr{F}_{X,\alpha}$
as a graded Stokes-filtered local system
by the filtration characterized by the following
equality:
\begin{align*}
    (\scr{F}_{X,\alpha})_{\leqslant c,e^{\mathrm{i}\theta}}
    =\bigoplus_{j \,:\, c_j\leqslant_\theta c}(\scr{F}_{X,\alpha}^j)_{e^{\mathrm{i}\theta}}\quad (\forall e^{\mathrm{i}\theta}\in S^1).
\end{align*}

The following is a well-known local comparison theorem:
\begin{lemma}\label{localrh}
We have an isomorphism of
Stokes-filtered local systems
\[\rh^{\mathrm{local}}\colon \scr{F}_{X,\alpha}\xrightarrow{\sim} \mathrm{Sol}_z(E_{X,\alpha}^0,\nabla^0)\]
locally given by $\rh^{\mathrm{local}}=\bigoplus_{j=1}^s \rh^{\mathrm{local}}_j$, where
$\rh^{\mathrm{local}}_j\colon \scr{F}_{X,\alpha}^j\to \mathrm{Sol}_z(E_{X,\alpha}^{0,j},\nabla^0)$
is defined by
\begin{align*}
    \langle\rh^{\mathrm{local}}_{j}([c]),[u^k_jdu_j]\rangle=e^{-c_j/z}
    \int_c e^{-f_j(u_j)/z}u_j^kdu_j
\end{align*}
for $[c]\in (\scr{F}_{X,\alpha}^j)_{e^{\i\theta}}=H_1(\widetilde{\P}^1_{u_j},\widetilde{D}^\theta_j;\C)$,
$j=1,\dots,s$, and $k=0,1,\dots,m_j-1$.
\end{lemma}
\begin{proof}
Although this is well known,
we shall give a sketch of proof to fix some notation
for later use. 
Take any $d\in \R$. 
For $j=1,\dots,s$ and $\ell=0,1,\dots,m_j$,
we set 
\[c^{(j)}_{\ell,d}(t)
=t\exp\left(\frac{2\pi\ell+d}{m_j+1}\i\right)\]
for $t\geq 0$ and 
\[c^{(j)}_{\ell,d}(t)
=-t\exp\left(\frac{2\pi(\ell+1)+d}{m_j+1}\i\right)\]
for $t\leq 0$. 
Then, we obtain a basis $[c_{\ell,d}^{(j)}]\in H_1(\widetilde{\P}^1_{u_j},
\widetilde{D}^{{\mathbb{I}_d}}_j;\C)$ $(\ell=0,1,\dots,m_j-1)$,
and 
   \begin{align*}
        &\int_{c^{(j)}_{\ell,d}} \exp(-f_j(u)/z)u_j^kdu_j=\exp({2\pi\i(k+1)\ell/(m_j+1)})h_{j,k}(z),\\
        &h_{j,k}(z)=(m_j+1)^{-1}({1-e^{2\pi\i(k+1)/(m_j+1)}})\{(m_j+1)z\}^{\frac{k+1}{m_j+1}}
        \Gamma\left(\frac{k+1}{m_j+1}\right)
    \end{align*}
for $|\arg(z)-d|<\pi/2$. 
Hence, if we define
\begin{align}\label{Cjld}
    C_{\ell,d}^{(j)}
    =\frac{1}{m_j+1}\sum_{\ell'=0}^{m_j} \exp\left(-\frac{2\pi\mathrm{i}(\ell+1)\ell'}{m_j+1}\right)c_{\ell',d}^{(j)}
\end{align}
for $\ell=0,1,\dots,m_j-1$,
then we have
\[\int_{C_{\ell,d}^{(j)}}\exp(-f_j(u_j)/z)u_j^kdu_j=\delta_{k,\ell}h_{j,k}(z),\]
which implies the lemma.
\end{proof}
Set $\scr{A}_\mu=\scr{A}_{\mu_\alpha}$
to simplify the notation.  We consider the tensor product
$\scr{G}_{X,\alpha}=\scr{A}_\mu\otimes \scr{F}_{X,\alpha}$,
which is naturally regarded as 
a graded $\scr{A}_\mu$-module.
The associated isomorphism 
of graded $\scr{A}_\mu$-modules
is also denoted by $\rh^\loc$:
\begin{align*}
    \rh^\loc\colon
    \scr{G}_{X,\alpha}
    \longrightarrow 
    \Sol_\mu
    (E^0_{X,\alpha},\nabla^0)
    =\scr{A}_\mu
    \Sol_z
    (E_{X,\alpha}^0,\nabla^0).
\end{align*}

\subsection{Betti object}\label{BO}
We consider real oriented blow-up
\[\varpi_X\colon \widetilde{X}_D\coloneqq \mathrm{Bl}_D^{\R}(X)\to X\]
along $D=\{p_1,\dots,p_r\}$.
For the local neighborhood $(V_k,v_k)$
of $p_k$, we have the explicit description
$\varpi_X^{-1}(V_k)
=\{(v_k,e^{\i\theta_k})\in V_k\times {S}^1\mid v_k=|v_k|e^{\i\theta_k}\}$.
Let $S^1_k$ denote the boundary $\varpi_X^{-1}(p_k)=\{e^{\i\theta_k}\mid \theta_k\in \R\}$.

Let $j\colon \C^*\times Y\to \widetilde{\C}_0\times \widetilde{X}_D$ and $i\colon {S}^1\times \widetilde{X}_D\to \widetilde{\C}_0\times \widetilde{X}_D$
be the inclusions.
Set $\widetilde{\scr{O}}_{{S}^1\times\widetilde{ X}_D}=i^{-1}j_*\scr{O}_{\C^*\times Y}$. 
Fix an open interval $I\subset {S}^1$.
Then we obtain a local system 
on $\widetilde{X}_D$ as follows:
For any open subset $U\subset \widetilde{X}_D$,
set 
\[\cal{M}_{X,\alpha}^I(U)=\scr{A}_\mu(I)\exp(-z^{-1}f)\subset \widetilde{\scr{O}}_{{S}^1\times \widetilde{X}_D}(I\times U),\]
which defines a local system of $\scr{A}_\mu(I)$-modules. 
Here, $f$ is regarded as a 
holomorphic function on $U\cap Y$
by taking a sheet $\widetilde{U}\subset \pi_Y^{-1}(U\cap Y)$.
Although a different choice of the sheet
defines a different function,
the resulting difference in $\exp(-z^{-1}f)$
is contained in $\scr{A}_\mu(I)$.
Hence the submodule $\scr{A}_\mu(I)\exp(-z^{-1}f)$ is well-defined. 

Recall that for each point 
$p_k\in D$,
we chose a local coordinate 
chart $(V_k,v_k)$
such that either 
$\alpha_{|V_k}=v_k^{-n_k}dv_k$
for $n_k>1$ or 
$\alpha_{|V_k}=\alpha_kv_k^{-1}dv_k$
for some $\alpha_k\in \C^*$ holds.

If $\alpha_{|V_k}=v_k^{-n_k}dv_k$
for $n_k>1$, then
on the corresponding boundary
$S^1_{k}$, we consider 
a subset $\widetilde{D}_k^I\subset S^1_k$ as follows:
\begin{align}\label{Irregular boundary}
    \widetilde{D}_k^I\coloneqq \{e^{\i\theta_k}\in S^1_k
    \mid \mathrm{Re}(e^{-\i((n_k-1)\theta_k+\theta)})<0 
    \text{ for all } e^{\i\theta}\in  I) \}.
\end{align}

If $\alpha_{|V_k}=\alpha_kv_k^{-1}dv_k$
for some $\alpha_k\in \C^*$, we set
\begin{align}\label{regular boundary}
    \widetilde{D}_k^I\coloneqq
    \begin{dcases}
        S^1_k&(\mathrm{Re}(e^{-\i\theta}\alpha_k)<0\text{ for all }e^{\i\theta}\in I),\\
        \emptyset &(\text{otherwise}).
    \end{dcases}
\end{align}

We also set $\widetilde{D}^I=\bigcup_k\widetilde{D}_k^I$
and $\widetilde{X}_D^I\coloneqq Y\cup \widetilde{D}^I\subset \widetilde{X}_D$.
Let $\scr{M}_{X,\alpha}^I$ denote
the restriction of $\cal{M}_{X,\alpha}^I$
to $\widetilde{X}^I_D$. 
Then, we consider the relative homology group
\[\scr{L}^{\mathrm{pre}}_{X,\alpha}(I)=  H_1(\widetilde{X}_D^I,\widetilde{D}^I;\scr{M}_{X,\alpha}^I)\] with local system coefficients, 
which is an $\scr{A}_\mu(I)$-module. 
Then, the correspondence
$I\mapsto \scr{L}^{\mathrm{pre}}_{X,\alpha}(I)$
together with the natural maps $\scr{L}^{\mathrm{pre}}_{X,\alpha}(I)\to\scr{L}_{X,\alpha}^{\rm pre}(J)$ 
for connected open subsets $J\subset I$
defines a presheaf 
of $\scr{A}_\mu$-modules
on ${S}^1$.
\begin{definition}[Global Betti $\scr{A}_\mu$-module]
    Let $\scr{L}_{X,\alpha}$ 
    denote the sheaf of $\scr{A}_\mu$-modules
    associated with 
    $\scr{L}_{X,\alpha}^{\mathrm{pre}}$. 
    We call $\scr{L}_{X,\alpha}$
    the global Betti $\scr{A}_\mu$-module
    of $(X,\alpha)$. 
\end{definition}
Since the monodromy of $\scr{M}_{X,\alpha}^I$ around $p_k$
with $n_k=1$ is $\exp(2\pi\i\alpha_k/z)$, 
we have 
$H_d(\widetilde{X}_D^I,\widetilde{D}^I;\scr{M}_{X,\alpha}^I)=0$ for $d\neq 1$.
It follows that 
the sheaf $\scr{L}_{X,\alpha}$
is a locally free $\scr{A}_\mu$-module
of rank 
$2g-2+\sum_{k=1}^{r}n_k$, 
which equals to
$\sum_{j=1}^s m_j$ by \eqref{genus}.

\begin{lemma}[Betti local-to-global isomorphism]\label{etaX}
For any $(C_f,\mu_\alpha)$-generic direction $d\in \R$,
    there is a morphism
    of $\scr{A}_\mu$-modules
    \begin{align*}
        \eta^{X,\alpha}_d\colon \scr{G}_{X,\alpha|{\mathbb{I}_d}}\longrightarrow  
        \scr{L}_{X,\alpha|{\mathbb{I}_d}}.
    \end{align*}
The definition of $\eta_d^{X,\alpha}$
is given by \eqref{etaeta} in the proof.
\end{lemma}
\begin{proof}
For a $(C_f,\mu_\alpha)$-generic direction $d$, the paths
\[\ell_c(t)=c+te^{\i d}\quad(c\in C_f+\mu_\alpha(L), t\geq 0)\]
do not intersect each other. 
It follows that 
the foliation \[\mathrm{Im}(e^{-\i d}f)=\text{constant}\]
does not have saddle connections, 
i.e., there is no
paths in the foliation which connect 
a pair of zero points.  

Under this observation,
we define a path
$\gamma_{\ell,d}^{(j)}$ by the following 
conditions:
\begin{enumerate}
    \item $\gamma^{(j)}_{\ell,d}\colon \R\to Y$
    is piecewise $C^\infty$ 
    with a unique non-smooth point $\gamma^{(j)}_{\ell,d}(0)=q_j$. 
    \item There is $\varepsilon_{j}>0$
    such that $\gamma_{\ell,d}^{(j)}((-\varepsilon_j,\varepsilon_j))\subset U_j$ for all $\ell=0,\dots,m_j$.
    Moreover, we have 
    \begin{align*}
        u_j(\gamma_{\ell,d}^{(j)}(t))=
            t\exp\left(\frac{2\pi  \ell+ d}{m_j+1}\i \right)
    \end{align*}
    for $0\leq t<\varepsilon_j$ and 
    \begin{align*}
        u_j(\gamma_{\ell,d}^{(j)}(t))=
            -t\exp\left(\frac{2\pi(\ell+1)+d}{m_j+1}\i \right) 
    \end{align*}
    for $-\varepsilon_j< t\leq 0$. 
    
    \item For $t\neq 0$, we have $\mathrm{Im}[e^{-\i d}\alpha (\gamma^{(j)}_{\ell,d*}(\partial_t))]=0$
    and $e^{-\i d}\alpha(\gamma^{(j)}_{\ell,d*}(\partial_t))>0$. 
\end{enumerate}

The path $\gamma^{(j)}_{\ell,d}$ is determined by these
conditions up to orientation-preserving reparameterization. 
We claim that 
\begin{align}\label{etaeta}
    \eta_d^{X,\alpha}([c_{\ell,d}^{(j)}])\coloneqq 
[\exp(-z^{-1}f)\otimes \gamma_{\ell,d}^{(j)}]
\end{align}
defines a section in $\scr{L}_{X,\alpha}({\mathbb{I}_d})$,
where $f$ is taken so that 
$f(\gamma_{\ell,d}^{(j)}(0))=c_j$. 

To see this, we consider the behavior of $\gamma_{\ell,d}^{(j)}(t)$ as $|t|\to \infty$.
We consider the behavior for $t\gg0$
(the behavior for $t\ll 0$ is similar).
By Assumption \ref{Assume} and 
\cite{boissy2015connected}*{Appendix A.1, Proposition 3.7}, orbits for 
the geodesic flow of $e^{-\i d}\alpha$
converges to a pole
for any $(C_f,\mu_\alpha)$-generic direction 
$d$.
It follows that 
there exists a unique $k\in\{1,\dots,r\}$
such that $\gamma_{\ell,d}^{(j)}(t)\in V_k$
for sufficiently large $t>0$.

Recall that 
either 
$\alpha_{|V_k}=v_k^{-n_k}dv_k$
for $n_k>1$ or 
$\alpha_{|V_k}=\alpha_kv_k^{-1}dv_k$
for some $\alpha_k\in \C^*$ holds. 
The foliation 
of these one forms are classically well-known and described explicitly in 
\cite{Strebel1984}*{\S 7.2, 7.3}
(in terms of the associated 
quadradic differentials). 

According to the description, 
when $\alpha_{|V_k}=v_k^{-n_k}dv_k$
for $n_k>1$,
it follows that 
$\gamma^{(j)}_{\ell,d}(\infty)\coloneqq  \lim_{t\to \infty}\gamma_{\ell,d}^{(j)}(t)$
in $\widetilde{X}_D$ is well defined.  
Moreover, 
for any $e\in \mathbb{I}_d$,
there exists an open neighborhood 
$I$ of $e$ such that 
$\gamma^{(j)}_{\ell,d}(\infty)\in \widetilde{D}^I$,
which implies the claim in this case. 

When $\alpha_{|V_k}=\alpha_kv_k^{-1}dv_k$
for some $\alpha_k\in \C^*$, the description of the foliation on 
$V_k$ in \cite{Strebel1984}*{\S 7.2} implies that 
$\gamma_{\ell,d}^{(j)}$
is a logarithmic spiral or a straight line.
In the straight line case,
we obtain the desired result. 
In the logarithmic spiral case,
we replace the 
path in $V_k$ with a straight line 
to $p_k$ (we may assume that $V_k$ is a small disk).
We then obtain the desired result.

In every case, 
we obtain a section in $\mathscr{L}_{X,\alpha}(\mathbb{I}_d)$,
which implies the lemma. 
\end{proof}

\subsection{Comparison of isomorphisms}\label{CI}
The following is the second main theorem of this paper:

\begin{theorem}[cf.{\cite{kontsevich2024holomorphic}*{Conjecture 4.7.1}}]\label{2nd main}
Let $(X,\alpha)$ satisfy Assumption $\ref{Assume}$.
\begin{itemize}
\item There exist a cyclic covering 
${\mathcal{I}}=\{I_k\}_{k\in \Lambda_K}$
and an isomorphism
\[{\Xi}_{X,\alpha,k}\colon E_{X,\alpha}\otimes\scr{A}_{|I_k}\longrightarrow E^0_{X,\alpha}\otimes \scr{A}_{|I_k}
\]
such that the tuple $((E_{X,\alpha},\mathcal{I},\Xi_{X,\alpha}),(E^0_{X,\alpha},\nabla^0))$
is an object of $\dR_{C_f,\mu_\alpha}$
and that the underlying formal isomorphism 
is $\widehat{\Xi}_{X,\alpha}$ defined in Theorem $\ref{Xihat}$.
\item There is an isomorphism
\begin{align*}
    \mathrm{rh}^{\rm global}\colon 
    \scr{L}_{X,\alpha}\longrightarrow \mathrm{Sol}_{\mu_\alpha}(E_{X,\alpha},{\Xi}_{X,\alpha})
\end{align*}
of $\scr{A}_\mu$-modules 
such that 
the diagram
\begin{align}\begin{split}\label{Diagram}
    \xymatrix@C=7em{
\scr{G}_{X,\alpha|{\mathbb{I}_d}}\ar[r]^{{\rm rh}^{\rm local}}\ar[d]_{\eta^{X,\alpha}_d}& \mathrm{Sol}_{\mu_\alpha }(E^0_{X,\alpha},\nabla^0)_{|{\mathbb{I}_d}}
\ar[d]^{\mathbb{D}_z(\Xi_{X,\alpha})^d}
\\
    \scr{L}_{X,\alpha|{\mathbb{I}_d}}\ar[r]^{{\rm rh}^{\rm global}}
    &\mathrm{Sol}_{\mu_\alpha}
    (E_{X,\alpha},{\Xi}_{X,\alpha})_{|{\mathbb{I}_d}}
     }
\end{split}
\end{align}
commutes for any $(C_f,\mu_\alpha)$-generic direction
$d\in\R$. 
Here, $\mathbb{D}_z(\Xi_{X,\alpha})^d$
denotes the unique 
analytic lift of $\mathbb{D}_z\widehat{\Xi}_{X,\alpha}$
over ${\mathbb{I}_d}$. 
In particular,
\begin{enumerate}
    \item the $\scr{A}_\mu$-module
    $\scr{L}_{X,\alpha}$ together with the filtration defined via $\mathrm{rh}^{\rm global}$ is an object of $\mathsf{Be}_{C_f,\mu_\alpha}$
    such that  $\scr{A}_\mu\otimes \gr\scr{L}_{X,\alpha}\simeq \scr{G}_{X,\alpha}$, and 
    \item the equality 
    $\gr(\mathrm{rh}^{\rm global})=\mathrm{rh}^{\rm local}$ holds. 
\end{enumerate}
\end{itemize}
\end{theorem}
\begin{proof}
For $(C_f,\mu_\alpha)$-generic direction $d\in \R$, $j=1,\dots,s$, and $k=0,1,\dots,m_j-1$, we set
\begin{align*}
    \Gamma_{k,d}^{(j)}
    =\frac{1}{m_j+1}\sum_{\ell=0}^{m_j}e^{-\frac{2\pi\i(k+1)\ell}{m_j+1} }\gamma^{(j)}_{\ell,d},
\end{align*}
where $\gamma^{(j)}_{\ell,d}$
denotes the path defined in the proof 
of Lemma \ref{etaX}.
Then, we set 
 \begin{align*}
\langle \mathbb{D}_z(\Xi_{X,\alpha})^d(\mathbb{D}_z[u_j^kdu_j]),[\omega] \rangle=
h_{j,k}^{-1}(z)\int_{\Gamma_{k,d}^{(j)}}\exp
\left(-z^{-1}\int_{q_j}^x\alpha\right)\omega
    \end{align*}
for $j=1,\dots,s$, $k=0,1,\dots,m_j-1$.
Here, $\{\mathbb{D}_z[u_j^kdu_j]\}_{j,k}$ denotes
the dual basis of $\{[u_j^kd u_j]\}_{j,k}$
in $\mathbb{D}_z(E^0_{X,\alpha})$,
$\omega\in H^0(X,\Omega_X(*D))$,
$[\omega]$ denotes the class in $E_{X,\alpha}$, and 
$h_{j,k}(z)$
denotes the function defined in 
Lemma \ref{localrh}.
We claim that 
$\mathbb{D}_z(\Xi_{X,\alpha})^d$
gives an isomorphism 
\[\mathbb{D}_z({\Xi}_{X,\alpha})^d\colon \mathbb{D}_z(E_{X,\alpha}^0)\otimes\scr{A}_{|{\mathbb{I}_d}}\longrightarrow \mathbb{D}_z(E_{X,\alpha})\otimes \scr{A}_{|{\mathbb{I}_d}}
\]
such that $\asy(\mathbb{D}_z(\Xi_{X,\alpha})^d)=
\mathbb{D}_z\widehat{\Xi}_{X,\alpha}$.
We then set 
\[\mathbb{D}_z\nabla^d\coloneqq[\mathbb{D}_z
({\Xi}_{X,\alpha})^d]^{-1} \circ  
\mathbb{D}_z
\nabla^0\circ \mathbb{D}_z({\Xi}_{X,\alpha})^d\]
where $\mathbb{D}_z\nabla^0$ denotes the connection on 
$\mathbb{D}_z(E^0_{X,\alpha},\nabla^0)$.
We then set 
\begin{align}\label{proofSol}
\Sol_\mu(E_{X,\alpha},\Xi_{X,\alpha})_{|{\mathbb{I}_d}}
\coloneqq 
\DR^\mu(\mathbb{D}_z(E_{X,\alpha}),\mathbb{D}_z\nabla^d)_{|{\mathbb{I}_d}},
\end{align}
although $\Xi_{X,\alpha}$ has not been defined yet.

We define a morphism
    $\mathrm{rh}^{\rm global}\colon \scr{L}_{|{\mathbb{I}_d}}\to \Sol_\mu(E_{X,\alpha},\Xi_{X,\alpha})_{|{\mathbb{I}_d}}$
    as follows: 
    For a representative $\omega\in H^0(X,\Omega_X^1(*D))$ of $[\omega]\in E_{X,\alpha}$
    and $[e^{-f/z}\otimes\gamma]\in \scr{L}_{X,\alpha}$, we set
    \begin{align*}
        \langle \mathrm{rh}^{\rm global}([e^{-f/z}\otimes \gamma]),[\omega]\rangle 
        =\int_\gamma e^{-f/z}\omega.
    \end{align*}
We then further claim that the diagram
\eqref{Diagram} commutes.

Assume that these two claims are true. 
Since the morphisms $\rh^{\loc}$ and $\mathbb{D}_z(\Xi_{X,\alpha})^d$
are isomorphisms, the morphism $\rh^\gl\circ \eta_d^{X,\alpha}$
is an isomorphism.
We define the filtration 
on $\scr{L}_{X,\alpha}$ via $\rh^{\gl}$.
Since $\mathbb{D}_z(\Xi_{X,\alpha})^d$
preserves the filtration,
$\eta_d^{X,\alpha}$ preserves the filtration.
Hence, we have
$\gr_c(\rh^\gl)\circ \gr_c(\eta_d^{X,\alpha})=
\gr_c(\rh^\loc)\circ \gr_c(\mathbb{D}_z(\Xi_{X,\alpha})^d)$.
This implies that $\gr(\eta_d^{X,\alpha})$
and $\gr(\rh^{\gl})$ are isomorphisms.
It follows that $\eta_d^{X,\alpha}$
is an isomorphism, hence 
$\rh^\gl$ is also an isomorphism. 

For two
$(C_f,\mu_\alpha)$-generic directions
$d$ and $d'$, the composition 
$\eta_d^{X,\alpha}\circ (\eta_{d'}^{X,\alpha})^{-1}$
corresponds to 
$[\mathbb{D}_z(\Xi_{X,\alpha})^{d}]^{-1}\circ
\mathbb{D}_z(\Xi_{X,\alpha})^{d'}$
on the intersection ${\mathbb{I}_d}\cap \mathbb{I}_{d'}$
via $\rh^\loc$ and $\rh^\gl$.
It follows that 
$\mathbb{D}_z(\Xi_{X,\alpha})^{d}\circ
[\mathbb{D}_z(\Xi_{X,\alpha})^{d'}]^{-1}$
satisfies the condition \eqref{aut}.
It also follows that 
\[\gr\left(\left(\eta_d^{X,\alpha}\right)^{-1}\circ \eta_{d'}^{X,\alpha}\right)=\id.\] 
Hence, for any 
$(C_f,\mu_\alpha)$-generic direction $d\in \R$, via $\eta_d^{X,\alpha}$,
$\scr{G}_{X,\alpha}$ is identified with
$\scr{A}_\mu\otimes \gr\scr{L}_{X,\alpha}$.
Hence $(\scr{L}_{X,\alpha},\scr{L}_{X,\alpha\leqslant })$ is an object of $\Be_{C_f,\mu_\alpha}$
with $\scr{G}_{X,\alpha}\simeq \scr{A}_\mu\otimes \gr\scr{L}_{X,\alpha}$.
Furthermore, taking a cyclic refinement 
$\mathcal{I}=\{I_k\}_{k\in\Lambda_K}$
of the covering 
\[\{{\mathbb{I}_d}\mid d:(C_f,\mu_\alpha)\text{-generic}\},\]
we obtain the object
$((E_{X,\alpha},\mathcal{I},\Xi_{X,\alpha}),(E^0_{X,\alpha},\nabla^0))\in\dR_{C,\mu}$.
The corresponding solution 
functor coincides with what was defined by 
\eqref{proofSol}. This completes the proof.

We now prove the two claims above. 
Take positive real numbers $\varepsilon_j$ and 
$R_j$ with $R_j>\varepsilon_j$
such that, for all $\ell$, $\gamma_{\ell,d}^{(j)}(t)\in U_j$
    for $|t|<\varepsilon_j$ and 
     $\gamma_{\ell,d}^{(j)}(t)\in \bigsqcup_k V_k$ for $|t|>R_j$.
Accordingly,
we divide the integral into three parts:
\begin{align*}
    \int_{\Gamma_{k,d}^{(j)}}\exp
\left(-z^{-1}\int_{q_j}^x\alpha\right)\omega
=
I_1(z)+I_2(z)+I_3(z).
\end{align*}
Here, $I_1(z)$, $I_2(z)$, and $I_3(z)$ denote
the integrals corresponding to the parameter
ranges
$|t|<\varepsilon_j$, $\varepsilon_j\leq |t|
\leq R_j$, and $|t|>R_j$,
respectively.

Let $J\subset \mathbb{I}_d$ be a
compact subset.  
    By standard arguments of 
   the stationary phase method, 
    $I_2(z)$ and $I_3(z)$
decay rapidly with Gevrey order one 
as $z\to 0$ for $\arg(z)\in J$. 
It remains to consider the asymptotic behavior
of $I_1(z)$. 
We use the notation 
$\Gamma_{k,d}(\varepsilon_j)$
for the sum of the paths restricted to
$|t|<\varepsilon_j$
so that 
\[ I_1(z)=\int_{\Gamma_{k,d}(\varepsilon_j)}
\exp
\left(-z^{-1}\int_{q_j}^x\alpha\right)\omega.\]
On $U_j$, we have $\omega=g_j(u_j)du_j$,
where 
$g_j(u_j)=\sum_n a_nu_j^n $
is a convergent power series.
    Then we have 
    \begin{align*}
        &\int_{\Gamma_{k,d}(\varepsilon_j)}
        \exp\left(-z^{-1}\int_{q_j}^x\alpha\right)
        g_j(u_j)du_j\\&=\sum_{\ell=0}^{m_j-1}
        \int_{\Gamma_{k,d}(\varepsilon_j)}
        \exp\left(-z^{-1}\int_{q_j}^x\alpha\right)
        g_{j,\ell}(u_j)u^\ell_jdu_j
    \end{align*}
where $g_{j,\ell}(u_j)=\sum_{n=0}^\infty a_{m_jn+\ell}u_j^{m_jn+\ell}$. 
Then, by the proof of Lemma 
\ref{localrh} and 
again by the standard arguments of the stationary phase method,
we obtain
\begin{align*}
    &\asy\left[h_{j,k}^{-1}(z)\int_{\Gamma_{k,d}(\varepsilon_j)}
        \exp\left(-z^{-1}\int_{q_j}^x\alpha\right)
        g_{j,\ell}(u_j)u^\ell_jdu_j\right]\\
 =&\asy\left[h_{j,k}^{-1}(z)\int_{C_{k,d}^{(j)}}
         e^{-f_j(u_j)/z}
         g_{j,\ell}(u_j)u^\ell_jdu_j\right]\\
=&\delta_{k,\ell}\widehat{g}_{j,k}(-z),
\end{align*}
where $\widehat{g}_{j,k}(z)\in\C\jump{z}_1$
is the series defined in Theorem \ref{Xihat}. 
This implies the equality $\asy(\mathbb{D}_z(\Xi_{X,\alpha})^d)=
\mathbb{D}_z\widehat{\Xi}_{X,\alpha}$.

We now prove 
the 
commutativity of the diagram \eqref{Diagram}.
Take $C_{k,d}^{(j)}\in \scr{G}_{X,\alpha}({\mathbb{I}_d})$
as in \eqref{Cjld}. 
Then, we 
have
\begin{align*}
    \langle \mathbb{D}_z({\Xi}_{X,\alpha})^d\circ \rh^{\loc}(C_{k,d}^{(j)}),
    [\omega]\rangle &=
    \langle\mathbb{D}_z({\Xi}_{X,\alpha})^d(
    e^{-c_j/z}h_{j,k}[u_j^kdu_j]),[\omega]\rangle\\
    &=e^{-c_j/z}h_{j,k}(z)\langle\mathbb{D}_z({\Xi}_{X,\alpha})^d(
    [u_j^kdu_j]),[\omega]\rangle\\
    &=e^{-c_j/z}\int_{\Gamma_{k,d}^{(j)}}\exp
\left(-z^{-1}\int_{q_j}^x\alpha\right)\omega
\end{align*}
On the other hand,
we have
\begin{align*}
    \rh^\gl\circ \eta_d^{X,\alpha}(C_{k,d}^{(j)})
    &=\rh^\gl\left(\sum_{\ell=0}^{m_j}[\exp(-f/z)\otimes \gamma_{\ell,d}^{(j)}]\right)\\
    &=\int_{\Gamma_{k,d}^{(j)}}e^{-f/z}\omega.
\end{align*}
Since we set
$f(\gamma_{\ell,d}^{(j)}(0))=c_j$
in the definition of $\eta_d^{X,\alpha}$,
the commutativity holds.  
\end{proof}

\begin{remark}
     Theorem \ref{2nd main} can be interpreted as a reformulated variant 
    of the conjecture of Kontsevich--Soibelman
    \cite{kontsevich2024holomorphic}*{Conjecture 4.7.1}
    in this setting. 
    See also \cite{kontsevich2024holomorphic}*{Remark 4.7.1}. 
    It is expected that 
    the higher-dimensional version of 
    this theorem holds for a more general class of pairs $(X,\alpha)$
    consisting of a projective complex manifold $X$ and a meromorphic closed 1-form $\alpha$ on $X$. 
\end{remark}

\subsection{An example}\label{Example G}
We give a motivating example of the theory using the notation
in this section.
\subsubsection{Setting}
 Fix a complex number $\lambda\neq 0$
 and $\log \lambda\in \C$ with $\exp(\log \lambda)=\lambda$. 
 We consider the projective line $X=\P^1$
 with non-homogeneous coordinate $x$.
 We set $\alpha_\lambda=-(\lambda-x)x^{-1}dx$.
 We have $Y=\C^*$, $D=\{0,\infty\}$, and 
 $Z=\{\lambda\}$.
 $L=H_1(Y,\Z)$ is identified with $\Z[c]$
 with generator $c$
 which satisfies 
 $\mu(c)\coloneqq 
 \mu_{\alpha_\lambda}(c)=-2\pi\i \lambda$. 
 The pair $(X,\alpha)$ satisfies Assumption 
 \ref{Assume}.

 The universal covering $\pi_Y\colon \widetilde{Y}\to Y$ is identified with
 $\exp\colon \C\to \C^*$. 
 Then, we set 
 $f\colon \widetilde{Y}=\C\to \C$
 by $f(y)=-\lambda y+e^y$ 
 and $C=C_f=\{\lambda-\lambda\log \lambda\}$,
 which is the critical value of $f$
 at $\log \lambda\in \C$. 
A direction $d\in \R$ is $(C,\mu)$-generic 
if and only if $d\neq \arg \lambda+\pi/2\pmod{\pi}$,
where $\arg\lambda=\mathrm{Im}(\log \lambda)$. 
We use the following notation:
\begin{align*}
     d&\in I_\lambda\coloneqq (\arg \lambda-\pi/2,\arg\lambda+\pi/2),\\
     d'&\in I'_\lambda\coloneqq (\arg \lambda, \arg\lambda+3\pi/2).
\end{align*}

The local coordinate 
$u=u_1$ around $\lambda$
is explicitly given by the equation
\begin{align}\label{Lambert-Stirling}
    u^2/2=
    (x-\lambda)-(\lambda\log x-\lambda\log \lambda)
\end{align}
up to the sign of $u$. 
The local coordinates $v_0$ and $v_\infty$
around $0$ and $\infty$
are taken so that 
$\alpha_\lambda=-\lambda v_0^{-1}dv_0$
and $
\alpha_\lambda=(v_\infty^{-2}+\lambda v_\infty^{-1})dv_\infty$. 

\subsubsection{de Rham objects}
The global de Rham structure is 
given as follows:
\begin{align*}
    E_{X,\alpha}&=
    \Cok\left[\C[x,x^{-1}]\otimes {\C\conv{z}}\xrightarrow{zd-(\lambda-x)x^{-1}dx}\C[x,x^{-1}]dx\otimes {\C\conv{z}}
    \right]\\
    &\simeq \C\conv{z}[x^{-1}dx].
\end{align*}
The local de Rham structure 
is given as follows:
\begin{align*}
    \widehat{E}^0_{X,\alpha}
    &=\Cok\left[\C\jump{u,z}[z^{-1}]
    \xrightarrow{zd+udu}\C\jump{u,z}[z^{-1}]du\right]\\
    &\simeq \C\pole{z}[du].
\end{align*}
We set $E_{X,\alpha}^0={\C\conv{z}}[du]$. 
Here, we have the relation 
\begin{align}\label{equiv}
[u^Ndu]=\begin{dcases}
    (-1)^{N/2}(N-1)!!z^{N/2}[du]
    &(N\equiv 0\pmod{2}),\\
    0&(N\equiv 1\pmod{2}).  
\end{dcases}
\end{align}
By the equation \eqref{Lambert-Stirling},
there exists an analytic function
$\xi_\lambda(u)$ in $u$ with $\xi_\lambda(0)=0$, 
\begin{align}\label{xi(u)}
x=\lambda(1+\xi_\lambda(u)),&&\xi_\lambda(u)-\log(1+\xi_\lambda(u))=\frac{u^2}{2\lambda}.
\end{align}
It follows that 
we have
\[x^{-1}dx=\frac{d\log(1+\xi_\lambda(u))}{du}du=
\left(\frac{d\xi_\lambda}{du}-\frac{u}{\lambda}\right)du.\]
By  \eqref{equiv},
we have
\begin{align}\label{XIGAMMA}
    \widehat{\Xi}_{X,\alpha}([x^{-1}dx])=\sum_{n=0}^\infty 
    (-1)^n(2n+1)!!\xi_{2n+1}(\lambda) z^{n}[du]
\end{align}
where $\xi_\lambda(u)=\sum_{n=0}^\infty \xi_{n}(\lambda)u^n$ denotes the 
Taylor expansion. 
By \eqref{xi(u)}, we may 
compute the first several terms of 
the coefficient of $[du]$ in \eqref{XIGAMMA}
explicitly as follows:
\begin{align*}
    \lambda^{-1/2}-\frac{1}{12}\lambda^{-3/2}z+\frac{1}{288}\lambda^{-5/2}z^2
    +\frac{139}{51840}\lambda^{-7/2}z^3-\cdots.
\end{align*}
We will see the
following later:
\begin{align}\label{Bernoulli}
    \widehat{\Xi}_{X,\alpha}([x^{-1}dx])
    =\lambda^{-1/2}\exp\left(-\sum_{n=1}^\infty \frac{B_{2n}}{2n(2n-1)\lambda^{2n-1}}z^{2n-1}\right)[du],
\end{align}
where $B_{2n}$ $(n=1,2,\dots)$ denote the Bernoulli numbers.
We put 
\[b_\lambda(z)=\sum_{n=1}^\infty \frac{B_{2n}}{2n(2n-1)\lambda^{2n-1}}z^{2n-1}\]
for later use. 
\subsubsection{Betti objects}
The real blow-up
$\widetilde{X}_D=\widetilde{\P}^1_{\{0,\infty\}}$
has two boundaries
$S^1_0$
and $S^1_\infty$. 
The set
$\widetilde{D}^I_{\infty}$ for $\infty$
is given by the formula \eqref{Irregular boundary},
and the set 
$\widetilde{D}^I_0$
is given by the formula \eqref{regular boundary}. 
The locally free $\scr{A}_\mu$-module $\scr{L}_{X,\alpha}$ is of rank one.

The local frame $[e^{-f/z}\otimes \gamma_d]$
(with abbreviation $\gamma_d\coloneqq \gamma_{0,d}^{(1)}$)
of $\scr{L}_{X,\alpha}$
described in the proof of Lemma \ref{etaX}
is given as follows:
If $d\in I_\lambda$,
the path $\gamma_d$ connects 
$\widetilde{D}^I_0$
and $\widetilde{D}^I_\infty$
for suitable $I\subset S^1$.  
If $d'\in I'_\lambda$,
the path $\gamma_{d'}$ is the 
Hankel contour with boundaries in 
$\widetilde{D}^I_\infty$
for suitable $I\subset S^1$.

The local Betti structure is given by the local system $\scr{F}_{X,\alpha}$, which is of rank one with monodromy $-1$. 
The boundary $S^1_u$ of
$\widetilde{\P}^1_u$
has the subset 
\[\widetilde{D}^\theta_u=\{e^{\i\theta_u}\in S^1_u\mid \mathrm{Re}(\exp(-2\pi\i(2\theta_u+\theta)))>0\},\]
which has two connected components. 
In particular, the path
\[c_d(t)=c^{(1)}_{0,d}(t)=t\exp(\pi\i d)
\quad (t\in \R)\]
defines a basis of 
$(\scr{F}_{X,\alpha})_{e^{\i\theta}}=H_1(\widetilde{\P}^1_u,\widetilde{D}_u^\theta)$
for suitable $\theta$.

\subsubsection{Riemann--Hilbert isomorphisms}
The correspondence of 
the local data 
is given by the 
well-known Gaussian integral:
\begin{align*}
    \langle \rh^\loc(C_d),[du]\rangle
    =\int_{c_d}e^{-u^2/z}du=\sqrt{2\pi z}
\end{align*}
for $|\arg z-d|<\pi/2$. In particular, we have
$h(z)=h_{1,0}(z)=\sqrt{2\pi z}$. 

The 
global isomorphism 
$\rh^\gl$
is given as follows.
\begin{align*}
    \langle\rh^\gl([e^{-f/z}\otimes \gamma_d]),[x^{-1}dx] \rangle
    =\int_{\gamma _d} e^{-x/z}x^{\lambda/z}
    \frac{dx}{x}. 
\end{align*}
It follows 
that we have 
\begin{align*}
\langle \mathcal{S}_d(\mathbb{D}_z(\widehat{\Xi}_{X,\alpha}))\mathbb{D}_z[du],[x^{-1}dx]\rangle
=\frac{
\exp((\lambda-\lambda\log \lambda)/z)z^{\lambda/z}\Gamma(\lambda/z)}{\sqrt{2\pi z}}
\end{align*}
when $d\in I_\lambda$.
Here, 
$\mathbb{D}_z[du]$ denotes the 
dual basis. 
Together with the Stirling's formula
    \begin{align*}
        \Gamma(s)\sim e^{-s}s^{s-1/2}(2\pi)^{1/2}\exp(b_1(s^{-1}))
        \quad (-\pi+\varepsilon<\arg s<\pi-\varepsilon,
        0<\varepsilon),
    \end{align*}
we obtain \eqref{Bernoulli}. 
We also obtain the following for $d'\in I'_\lambda$:
\begin{align*}
    \langle \mathcal{S}_{d'}(\mathbb{D}_z(\widehat{\Xi}_{X,\alpha}))\mathbb{D}_z[du],[x^{-1}dx]\rangle
=
\sqrt{2\pi}\frac{z^{\lambda/z-1/2}\i\exp((\lambda-\lambda\log \lambda-\pi\i \lambda)/z)}{\Gamma(1-\lambda/z)}.
\end{align*}

\subsubsection{Explicit descriptions of connections}
Set $\Xi_{X,\alpha}^d=\cal{S}_d(\widehat{\Xi}_{X,\alpha})$ for $d\in I_\lambda$
and $\Xi_{X,\alpha}^{d'}=\cal{S}_{d'}(\widehat{\Xi}_{X,\alpha})$
for $d'\in I'_\lambda$. 
We also set 
$\nabla^{d}\coloneqq 
(\Xi_{X,\alpha}^d)^{-1}\nabla^0\Xi_{X,\alpha}^d$ for $d\in I_\lambda$ and define 
$\nabla^{d'}$ for $d'\in I'_\lambda$ similarly. 
For $d\in I_\lambda$,
we have
\begin{align*}
\nabla^{d}=
d+[\lambda(\psi(\lambda/z)-\log(\lambda/z) )-(\lambda-\lambda\log\lambda)]\frac{dz}{z^2},
\end{align*}
where  
$\psi(s)=d\log\Gamma(s)/ds$ denotes the digamma function, and we set \[-\pi<\arg(\lambda/z)<\pi\]
for the branch of $\log(\lambda/z)$.  
    For $d'\in I_\lambda'$, we have
\begin{align*}
    \nabla^{d'}=d+[\lambda(\psi(1-\lambda/z)-\log(\lambda/z)-\pi \i)-(\lambda-\lambda\log \lambda)]\frac{dz}{z^2},
\end{align*}
where we set $0<\arg(\lambda/z)<2\pi$
for the branch of $\log(\lambda/z)$.

\begin{remark}
One can directly check that the connection forms of $\nabla^d$ and $\nabla^{d'}$ have the asymptotic expansion
of the form 
\begin{align*}
    \widehat{\nabla}=d-\frac{1}{2}\left( 1+\sum_{n=1}^\infty \frac{B_{2n}}{n\lambda^{2n-1}}z^{2n-1}\right)\frac{dz}{z}-(\lambda-\lambda\log\lambda)\frac{dz}{z^2}
\end{align*}
by using the well-known asymptotic expansion
of the digamma function. 
\end{remark}

The connection formula of the gamma function
implies the following: 
    \begin{align*}
        \Xi_{X,\alpha}^d\circ
        (\Xi_{X,\alpha}^{d'})^{-1}=
        \begin{cases}
            1-u&(\text{on }J_+)\\
            1-u^{-1}&(\text{on }J_-),
        \end{cases}
    \end{align*}
    where we set $u=\exp(2\pi\i\lambda/z)$ and 
    $J_{\pm}=\{z\in {S}^1\mid \pm\mathrm{Im}(\lambda/z)>0\}$. 

\begin{remark}
    One can directly check that 
    we have 
    $g_+\nabla^d=\nabla^{d'}g_+$ on $J_+$
    and 
    $g_- \nabla^d=\nabla^{d'}g_-$
    on $J_-$
    for $g_\pm =1-u^{\pm 1}$ using the classical connection formula 
    for the digamma function. 
\end{remark}

\begin{appendix}
    \section{The surjectivity of the map \eqref{I0}}\label{Appendix}
In this appendix,
we prove the surjectivity of the 
map \eqref{I0}.
\subsection{ Fréchet norms 
for sections of $\scr{A}_\mu$}
Let $I\subset  {\mathbb{I}_d}$
be an open arc for some $(C,\mu)$-generic $d$.  
Set 
\[\mathrm{Cpt}({I})\coloneqq 
\{J\subset {I}\mid J \text{ is compact.}\}\]
and consider the set of  maps
\[
\Exp({I})=\{
\mathsf{a}\colon \mathrm{Cpt}({I})\to \R, J\mapsto \mathsf{a}_J\mid 
J\subset J' 
\Longrightarrow\mathsf{a}_{J}\geq \mathsf{a}_{J'}\}.\]
There is a natural inclusion
$\R\hookrightarrow\mathrm{Exp}({I})$
by the constant map.
We also define a partial order as follows: $\mathsf{a}\leq \mathsf{b}$ if and only if 
$\mathsf{a}_J\leq \mathsf{b}_J$
for each $J\in \mathrm{Cpt}({I})$.
The sum $(\mathsf{a}+\mathsf{b})_J\coloneqq 
\mathsf{a}_J+\mathsf{b}_J$
is also defined on $\Exp({I})$. 
Let $\Exp^+({I})$ denote 
the subset of $\Exp({I})$ 
whose elements $\mathsf{a}$
satisfy $\mathsf{a}_J>0$
for any $J$.

Fix a connected open subset 
$\widetilde{I}\subset \widetilde{\C}$
such that $\widetilde{I}\cap S^1=I$.
Consider closed arcs of the form
$J=\{e^{\i\theta}\in S^1\mid |\theta-d'|\leq \Theta\}\subset I$ for $d'\in\R$ and $\Theta>0$. 
Then we use notation 
$\overline{S}(J,\rho)=\overline{S}(d,\Theta,\rho)$ for $\rho>0$
(see notation in \S \ref{Preliminary}).
We set 
\[\CS(\widetilde{I})
\coloneqq\{K=\overline{S}(J,\rho)\subset \widetilde{I}\mid J\in\Cpt(I) \text{ is a closed arc}, \rho>0\}.\]

For $\mathsf{a}\in \mathrm{Exp}({I})$,  
$K=\overline{S}(J,\rho)\in \CS(\widetilde{I})$, 
and a section $f\in \jmath_*\scr{O}_{\C^*}(\widetilde{I})$, 
we set
\begin{align*}
    \|f\|_{\mathsf{a},K}&\coloneqq \sup_{z\in K}
    \exp\left(\mathsf{a}_J/|z|\right)|f(z)|.
\end{align*}   
Then, for each $\mathsf{a}\in \mathrm{Exp}({I})$, we set 
    \begin{align*}
\widetilde{\scr{O}}(\widetilde{I})_{\mathsf{a}}
&\coloneqq 
\{f\in \jmath_*\scr{O}(\widetilde{I})
\mid 
\forall K\in \mathrm{CS}(\widetilde{I}),
\|f\|_{\mathsf{a},K}<\infty\}.
\end{align*}
The image of 
the restriction map
$\widetilde{\scr{O}}(\widetilde{I})_{\mathsf{a}}
\to \widetilde{\scr{O}}(I)$
is denoted by the same symbol. 
The family of semi-norms
$\{\|\cdot \|_{\mathsf{a},K}\mid K\in \CS(\widetilde{I})\}$
together with the usual sup norm on compact subsets 
in $\widetilde{I}\setminus S^1$
defines a Fréchet topology on 
$\widetilde{\scr{O}}(\widetilde{I})_{\mathsf{a}}.$
If $\mathsf{a}\leq \mathsf{b}$, 
then we have $\widetilde{\scr{O}}(\widetilde{I})_\mathsf{a}\supset 
\widetilde{\scr{O}}(\widetilde{I})_\mathsf{b}$.
We also note that 
if $\mathsf{a}\in \Exp^+({I})$,
then we have $\widetilde{\scr{O}}(\widetilde{I})_\mathsf{a}
\subset\scr{A}^{<0}(I)$. 
We also have 
$\widetilde{\scr{O}}(\widetilde{I})_\mathsf{a}\cdot 
\widetilde{\scr{O}}(\widetilde{I})_\mathsf{b}
\subset\widetilde{\scr{O}}(\widetilde{I})_\mathsf{a+b}$
for $\mathsf{a,b}\in \Exp({I})$. 
\begin{lemma}\label{App Lem 1}
    For positive $\varrho<1$ and an
    open arc $I\subset  {\mathbb{I}_d}$ as above, 
    there exists a pair $(\widetilde{I},\mathsf{a})$ of an open subset 
    $\widetilde{I}=\widetilde{I}(\varrho,I)$
    with $\widetilde{I}\cap S^1=I$
    and 
    $\mathsf{a}\in \Exp^+({I})$
    such that 
    \begin{align*}
        \scr{A}_{\mu,\varrho}^{<0}(I)\subset
        \widetilde{\scr{O}}(\widetilde{I})_{\mathsf{a}}.
    \end{align*}
\end{lemma}
\begin{proof}
Take $\varrho>0$ and $I$
as in the claim. 
Choose a positive real number $\mathsf{m}$ such that   
$\mathsf{m}<\mathsf{m}_L
\coloneqq \min_{\gamma\in L}\|\gamma\|$.  
Define a map 
$\varepsilon\colon \Cpt(I)\to \R_{>0}$,
$J\mapsto \varepsilon_J$
by
\begin{align*}
    \varepsilon_J\coloneqq 
    -\max\{\cos(\arg(\mu(\gamma))-\theta)
    \mid \gamma\in L,\mu(\gamma)<_I0,
    \theta\in J\}.
\end{align*}

 For each $J\in \Cpt(I)$, let $J^\circ$
 denote the interior and 
put \[\widetilde{I}_J\coloneqq \{(z,w)\in \widetilde{\C}\mid
w\in J^\circ, |z|< -(\log \varrho)^{-1} 
\varepsilon_JR (1-\mathsf{m}/\mathsf{m}_L)\}.\]
Set $\widetilde{I}=\bigcup_{J\in \Cpt(I)}\widetilde{I}_J$. 
We then take $\mathsf{a}\in \Exp^+({I})$
with the following property:
For any $J\in \Cpt({I})$,
there is $ J'\in\Cpt(I)$
such that $(J')^\circ \supset J $ and 
$\mathsf{a}_J<\varepsilon_{J'}R\mathsf{m}$.

By the 
inequality \eqref{inequality for A},
for any $f(z)=\sum_{\mu(\gamma)<_I}a_\gamma \exp(\mu(\gamma)/z)\in \scr{A}_{\mu,\varrho}^{<0}(I)$
and any $K=\overline{S}(J,\rho)\in \CS(\widetilde{I})$,
taking $J'$ with 
$J\subset (J')^\circ$ and $\mathsf{a}_J<\varepsilon_{J'}R\mathsf{m}$,
we have
\begin{align*}
    |f(z)|&\leq \sum_{\mu(\gamma)<_I 0 }
    |a_\gamma|\exp(-\varepsilon_{J'} R\|\gamma\|/|z|)\\
    &= \exp(-\varepsilon_{J'}
    R \mathsf{m}/|z|)
    \sum_{\mu(\gamma)<_I0}|a_\gamma|\exp
    (-\varepsilon_{J'}
    R\|\gamma\|
    (1-\mathsf{m}/\|\gamma\|)/|z| )\\
    &\leq \exp(-\mathsf{a}_J/|z|)
    \sum_{\mu(\gamma)<_I 0}|
    a_\gamma|\varrho^{\|\gamma\|}
    =\exp(-\mathsf{a}_J/|z|)\|f\|_{I,\varrho}
\end{align*}
for $z\in K\cap \widetilde{I}_{J'}$.
Hence $\scr{A}^{<0}_{\mu,\varrho}(I)\subset 
\widetilde{\scr{O}}(\widetilde{I})_{\mathsf{a}}$. 
\end{proof}
For a complex number $c$,
we set 
\begin{align*}
    \scr{A}_{\mu,\varrho}^{<c}(I)\coloneqq
    \sum_{\mu(\gamma)\leqslant_I c}
    \exp(\mu(\gamma)/z)\scr{A}_{\mu,\varrho}^{< 0}(I).
\end{align*}
We have $\scr{A}_{\mu}^{<c}(I)=\bigcup_{\varrho>0}\scr{A}_{\mu,\varrho}^{<c}(I)$. 
For $c\in \C$,
we define
$\cpx(c)\in \Exp(I)$
as follows: For $J\in \Cpt(I)$, we set 
\[\cpx(c)_J\coloneqq 
-\max_{e^{\i\theta }\in J}
\mathrm{Re}(e^{-\i\theta }c).\]
\begin{corollary}
    For $c\in \C$, 
    we have 
        $\scr{A}_{\mu,\varrho}^{<c}(I)
        \subset \widetilde{\scr{O}}(\widetilde{I})_{\mathsf{a}+\mathsf{cpx}(c)}.$\qed
\end{corollary}

\subsection{Fréchet norms 
for sections of
graded $\scr{A}_\mu$-modules}
We use the notation $C=\{c_1,\dots, c_m\}$.
Let $\scr{G}$ be a 
graded $\scr{A}_\mu$-module
in $\Be_{C,\mu}$. 
We have a decomposition 
\[\scr{G}=\bigoplus_{j=1}^m \scr{G}_j\]
where $\scr{G}_j$ denotes
the summand such that $\gr_c\scr{G}_j=0$
for $c\notin c_j+\mu(L)$. 
 
For $\varrho>0$ with $\varrho<1$
and $c\in C+\mu(L)$,
set 
\begin{align*}
    H^0(I,(\scr{G}_{j})_{<c})_\varrho\coloneqq 
    \scr{A}_{\mu,\varrho}^{<c-c_j}(I)\otimes_\C
    H^0(I,\gr_{c_j}\scr{G}).
\end{align*}
 We have $ H^0(I,(\scr{G}_j)_{<c})=\bigcup_{\varrho>0}H^0(I,(\scr{G}_j)_{<c})_\varrho$. 
Fix norms $\|\cdot \|_j$ on $H^0(I,\gr_{c_j}\scr{G})$.
We define semi-norms
\[ 
\{\|\cdot\|_{j,c, \mathsf{b},K}\mid  
\mathsf{b}\in \Exp(I), \mathsf{b}\geq0,
K\in \CS(\widetilde{I})\}
\]
on subspaces of $H^0(I,(\scr{G}_j)_{<c})_{\varrho}$
as follows:
For a section $g=\sum_{k} g_k(z) v_k$,
with 
$g_k(z)\in\scr{A}_{\mu,\varrho}^{<c-c_j}(I)$
and $v_k\in H^0(I,\gr_{c_j}\scr{G})$,
$K=\overline{S}(J,\rho)\in \CS(\widetilde{I})$,
and $\mathsf{b}\in \Exp({I})$
with $\mathsf{b}\geq 0$,
we set 
\begin{align*}
    \|g\|_{j,c,\mathsf{b},K}
    &\coloneqq \sup_{z\in K}\exp((\mathsf{a}+\mathsf{b}+\cpx(c-c_j))_J/|z|)\left\|\sum_{k}g_k(z)v_k
    \right\|_j \\
    &\leq \sum_k \|g_k\|_{\mathsf{a+b+cpx}(c-c_j),K}\|v_k\|_j. 
\end{align*}
We then set 
$\|\cdot\|_{\scr{G},c,\mathsf{b},K}
=\sum_j\|\cdot\|_{j,c,\mathsf{b},K}$
on $H^0(I,\scr{G}_{<c})_{\varrho}
\coloneqq \bigoplus_j H^0(I,(\scr{G}_{j})_{<c})_\varrho$.
We then obtain the subspaces
indexed by $\mathsf{b}\in\Exp(I)$:
\begin{align*}
    H^0(I,\scr{G}_{<c})_{\varrho,\mathsf{b}}
    \coloneqq \{g\in H^0(I,\scr{G}_{<c})_{\varrho}\mid 
    \|g\|_{\scr{G},c,K,\mathsf{b}}<\infty\}. 
\end{align*}
Each $H^0(I,\scr{G}_{<c})_{\varrho,\mathsf{b}}$
is a Fréchet space in a natural way.

The sheaf 
$\scr{E}nd(\scr{G})=\homscr(\scr{G},\scr{G})$ is also a
graded $\scr{A}_\mu$-module. 
For $0<\varrho<1$, set 
\begin{align*}
    H^0(I,\scr{E}nd^{<0}(\scr{G}))_\varrho
    \coloneqq 
    \bigoplus_{1\leq i,j\leq m}
    \scr{A}_{\mu,\varrho}^{<c_j-c_i}
    (I)\otimes H^0(I,
    \homscr
    (\scr{G}_{c_j},\scr{G}_{c_i})). 
\end{align*}
We have 
$H^0(I,\scr{E}nd^{<0}(\scr{G})) 
=\bigcup_{0<\varrho<1}H^0(I,\scr{E}nd^{<0}(\scr{G}))_\varrho
$. 
\begin{lemma}
    For 
    a section
    $\Psi\in H^0(I,\scr{E}nd^{<0}(\scr{G}))_\varrho$
    and $\mathsf{b}\in \Exp({I})$
    with $\mathsf{b}\geq 0$,
    we have  
    \begin{align*}
\Psi(H^0(I,\scr{G}_{<c}
)_{\varrho,\mathsf{b}})
\subset H^0(I,\scr{G}_{<c}
)_{\varrho,\mathsf{a+b}}.
    \end{align*}
\end{lemma}
\begin{proof}
By definition,
we have an expression
\[\Psi=\sum_{i,j=1}^m \sum_{k=1}^{n_{ij}}\psi_{ijk}(z)\otimes \Psi_{ijk}\]
where $\psi_{ijk}(z)\in \scr{A}_{\mu,\varrho}^{<c_i-c_j}(I)$
and $\Psi_{ijk}\in H^0(I,\homscr(\scr{G}_j,\scr{G}_i))$.
Take any section 
$g\in H^0(I,\scr{G}_{<c}
)_{\varrho,\mathsf{b}}$.
We have the expression
\[g=\sum_{j=1}^m 
\sum_{\ell} g_{j,\ell}(z)v_{j,\ell},\]
with 
$g_{j,\ell}(z)\in
\scr{A}_{\mu,\varrho}^{c-c_j}(I)$
and $v_{j,\ell}
\in H^0(I,\gr_{c_j}\scr{G})$. 
Then, the image
$\Psi(g)=\sum_i\Psi(g)_i$,
$\Psi(g)_i\in 
H^0(I,(\scr{G}_i)_{<c})_{\varrho}$
is expressed as follows:
\begin{align*}
    \Psi(g)_i=\sum_{j=1}^m \sum_{k,\ell}\psi_{ijk}(z)g_{j,\ell}(z)
    \Psi_{ijk}(v_{j,\ell}).
\end{align*}
By assumption,
we have
$\psi_{ijk}(z)g_{j,\ell}(z)
\in \widetilde{\scr{O}}(\widetilde{I})_{2\mathsf{a}+\mathsf{b}+
\mathsf{cpx}(c-c_i)}$,
which implies the claim. 
\end{proof}

\begin{corollary}\label{Appcor}
    For $\Psi\in 
    H^0(I,\scr{E}nd^{<0}
    (\scr{G}))_{\varrho}$
    and 
    $g\in H^0(I,\scr{G}_{<c})_\varrho$,
    the infinite sum 
    $h_\Psi(g)\coloneqq 
    \sum_{n=0}^\infty \Psi^n(g)$
    converges in the Fréchet topology. \qed
\end{corollary}

\subsection{Proof of the surjectivity}
For an element $c'$ with 
$c'<_{I_0}c$,
we have $c'<_{I_1}c$
or $c'<_{I_2}c$. 
We can take a partition 
\begin{align*}
    T_c(I_0)\coloneqq \{c'\in C(\scr{L})\mid c'<_{I_0} c\}
    =T_1\sqcup T_2
\end{align*}
so that $T_1\subset T_c(I_1)$ and $T_2\subset T_c(I_2)$. 
Then we have the map 
\[s_i\colon H^0(I_0,
    (\scr{A}_\mu\otimes\gr \scr{L})_{<c})
    \to H^0(I_i,
    (\scr{A}_\mu\otimes\gr \scr{L})_{<c})\]
so that $v=s_1(v)_{|I_0}+s_2(v)_{|I_0}$
and non-zero coefficients 
in $s_i(v)$ are contained in $\gr_{c'}\scr{L}$
for some $c'\in T_i$.
Take any $g\in H^0(I_0,
    (\scr{A}_\mu\otimes\gr \scr{L})_{<c})$.
There exists 
$\varrho>0$
such that $g\in H^0(I_0, (\scr{A}_{\mu,\varrho}\otimes\gr \scr{L})_{<c})$
and $\varphi$ is defined over $\scr{A}_{\mu,\varrho}(I)$.
Consider the map $\Phi$ defined by 
 \[\Phi(g)=g-(\varphi(s_2(g)_{|I_0})+s_1(g)_{|I_0})
 =(\id-\varphi)(s_2(g)_{|I_0})\] 
Since $(\id-\varphi)\in H^0(I,\scr{E}nd^{<0}(\scr{A}_\mu\otimes\gr\scr{L}))_\varrho$,
by Corollary \ref{Appcor}, the sum
\[h=h_\Phi(g)\coloneqq \sum_{n=0}^\infty \Phi^n(g)\]
converges. 
Under the map \eqref{I0},
we have 
$(-s_1(h),s_2(h))\mapsto g$, 
which implies the 
claim. \qed

\end{appendix}

\bibliographystyle{alpha}
\bibliography{summable}
\end{document}